\documentclass[11pt, notitlepage]{article}
\usepackage{amssymb,amsmath,comment}
\usepackage{amsthm}
\catcode`\@=11 \@addtoreset{equation}{section}
\def\thesection{\arabic{section}}

\def\theequation{\thesection.\arabic{equation}}
\catcode`\@=12
\usepackage{colortbl}
\usepackage{hyperref}
\usepackage[mathscr]{eucal}
\usepackage{epsf}
\usepackage{esint}
\usepackage{a4wide}
\def\R{\mathbb{R}}

\newcommand{\e}{\epsilon}

\newcommand{\noi} {\noindent}

\newcommand{\Tail} {\mathrm{Tail}}
\newcommand{\loc} {\mathrm{loc}}

\usepackage[all]{xy}
\catcode`\@=11
\def\theequation{\@arabic{\c@section}.\@arabic{\c@equation}}
\catcode`\@=12

\newtheorem{Theorem}{Theorem}[section]
\newtheorem{Lemma}[Theorem]{Lemma}
\newtheorem{prop}[Theorem]{Proposition}
\newtheorem{Corollary}[Theorem]{Corollary}
\newtheorem{Remark}[Theorem]{Remark}
\newtheorem{Definition}[Theorem]{Definition}

\newtheorem*{ack}{Acknowledgements}

\begin{document}

{\vspace{0.01in}}

\title
{\sc Higher H\"older regularity for 	mixed local and nonlocal degenerate elliptic equations}
\author{Prashanta Garain and Erik Lindgren}

\maketitle

\begin{abstract}\noindent
We consider equations involving a combination of local and nonlocal degenerate $p$-Laplace operators. The main contribution of the paper is almost Lipschitz regularity for the homogeneous equation and H\"older continuity with an explicit H\"older exponent in the general case. For certain parameters, our results also imply H\"older continuity of the gradient. In addition, we establish existence, uniqueness and local boundedness. The approach is based on an iteration in the spirit of Moser combined with an approximation method.\normalcolor

\medskip

\noi {Keywords: Mixed local and nonlocal $p$-Laplace equation, degenerate elliptic equations, existence and uniqueness, local boundedness, H\"older continuity, Moser iteration.}

\medskip

\noi{\textit{2020 Mathematics Subject Classification: 35B65, 35D30, 35J70, 35R09, 35R11.}}
\end{abstract}
\tableofcontents

\section{Introduction}
\subsection{Overview}
\nocite{BL}
In this article, we study regularity properties of weak solutions of the mixed local and nonlocal $p$-Laplace equation\normalcolor
\begin{equation}\label{maineqn}
-\Delta_p u+(-\Delta_p)^{s}u=f\text{ in }\Omega,
\quad 
\end{equation}
where $\Omega$ is an open and bounded set in $\mathbb{R}^N,\,N\geq 1$. We assume that $0<s<1,\,2\leq p<\infty$ and $f\in L^q_{\rm loc}(\Omega)$ for some $q\geq 1$ (for the precise assumptions, see Sections \ref{sec:main} and \ref{sec:prel}). Here
\begin{equation}\label{plap}
\Delta_p u=\text{div}(|\nabla u|^{p-2}\nabla u),
\end{equation}
is the $p$-Laplace operator and  
\begin{equation}\label{fracplap}
(-\Delta_p)^{s}u(x)=\text{P.V.}\int_{\mathbb{R}^N}\frac{|u(x)-u(y)|^{p-2}(u(x)-u(y))}{|x-y|^{N+ps}}\,dy,
\end{equation}
is the fractional $p$-Laplace operator, where P.V. denotes the principal value.

 The main objective of this article is to establish H\"older regularity of weak solutions of equation \eqref{maineqn}, with an explicit H\"older exponent. This is done in  Theorem \ref{teo:1} and Theorem \ref{nonthm}. From this, H\"older regularity of the gradient follows in the case when $f=0$ and $sp<(p-1)$. We also establish existence and uniqueness  in Theorem \ref{prop:death!} and local boundedness in Theorem \ref{teo:loc_bound}. Our results are presented in detail in the next section.

\subsection{Main results}\label{sec:main}
Here we present the main results of this paper: existence, uniqueness and regularity of weak solutions. For the notion of weak solutions and relevant notation such as $\mathrm{Tail}_{p-1,s\,p,s\,p}$, we refer to Section \ref{sec:prel}. In the theorem below, $p^*$ refers to the Sobolev exponent, see \eqref{eq:sobexp}.
\begin{Theorem}[Existence and uniqueness]
\label{prop:death!}
Suppose $1<p<\infty$, $0<s<1$ and $A>0$. Let $\Omega\Subset \Omega'\subset\mathbb{R}^N$ be two open and bounded sets where $f\in L^q(\Omega)$, with
\[
q\ge (p^*)'\quad \mbox{ if } p\ne N\qquad \mbox{ or }\qquad q>1\quad \mbox{ if }p=N,
\]
and $g\in W^{1,p}(\Omega')\cap L^{p-1}_{sp}(\mathbb{R}^N)$. Then there is a unique weak solution $u\in W_g^{1,p}(\Omega)$  of
\[
\left\{\begin{array}{rcll}
-\Delta_p u+A(-\Delta_p)^s\,u&=&f,&\mbox{ in }\Omega,\\
u&=&g,& \mbox{ in }\mathbb{R}^N\setminus \Omega.
\end{array}
\right.
\]
\end{Theorem}
\begin{Theorem}[Local boundedness]
\label{teo:loc_bound}
Suppose $2\leq p<\infty$, $0<s<1$ and $0\leq A\leq 1$. Let $\Omega\subset\mathbb{R}^N$ be an open and bounded set and $u\in W^{1,p}_{\rm loc}(\Omega)\cap L^{p-1}_{s\,p}(\mathbb{R}^N)$ be a weak solution of 
$$
(-\Delta_p) u+A(-\Delta_p)^s u = f \qquad \mbox{ in }\Omega,
$$
where $f\in L_\text{loc}^{q}(\Omega)$ with
\[
\left\{\begin{array}{lr}
q>\dfrac{N}{p},& \mbox{ if } p\le N,\\
&\\
q\ge 1,& \mbox{ if }p>N.
\end{array}
\right.
\]
Then $u^+=\max\{u,0\}$, satisfies $u^+\in L^\infty_\text{loc}(\Omega)$ and for every $0<R<1$ such that
$B_{R}(x_0)\Subset \Omega$ and every $0<\sigma<1$, there holds
\begin{equation}\label{bdest}
\begin{split}
&\|u^+\|_{L^\infty(B_{\sigma R}(x_0))}\\
&\quad\le C\,\left[\left(\fint_{B_{R}(x_0)} (u^+)^p\, dx\right)^\frac{1}{p}+\mathrm{Tail}_{p-1,s\,p,s\,p}(u^+;x_0, R)+\left(R^{p-\frac{N}{q}}\,\|f\|_{L^q(B_{R}(x_0))}\right)^\frac{1}{p-1}\right],
\end{split}
\end{equation}
where $C=C(N,s,p,q,\sigma)>0$.
\end{Theorem}

\begin{Theorem}[Almost Lipschitz regularity]
\label{teo:1}
Suppose $2\leq p<\infty$, $0<s<1$ and $0\leq A\leq 1$. Let $\Omega\subset\mathbb{R}^N$ be an open and bounded set and $u\in W^{1,p}_{\rm loc}(\Omega)\cap L^{p-1}_{s\,p}(\mathbb{R}^N)$ be a weak solution of 
\[
-\Delta_p u+A(-\Delta_p)^s u=0\qquad \mbox{ in }\Omega.
\] 
Then $u\in C^\delta_{\rm loc}(\Omega)$ for every $0<\delta<1$. 
\par
More precisely, for every $0<\delta<1$ and every ball $B_{2R}(x_0)\Subset\Omega$ with $0<R<1$, there exists a constant $C=C(N,s,p,\delta)>0$ such that
\begin{equation}
\label{apriori}
[u]_{C^\delta(B_{R/2}(x_0))}\leq \frac{C}{R^\delta}\,\left(\|u\|_{L^\infty(B_{R}(x_0))}+\mathrm{Tail}_{p-1,s\,p,s\,p}(u;x_0,R)\right).
\end{equation}
\end{Theorem}

\begin{Theorem}[Higher H\"older regularity]
\label{nonthm}
Suppose $2\leq p<\infty$, $0<s<1$ and $0\leq A\leq 1$. Let $\Omega\subset\mathbb{R}^N$ be an open and bounded set, and $f\in L^q_{\text{loc}}(\Omega)$ where
\[
\left\{\begin{array}{lr}
q>\dfrac{N}{p},& \mbox{ if } p\le N,\\
&\\
q\ge 1,& \mbox{ if }p>N.
\end{array}
\right.
\]
Let $\Theta=\min\{(p-N/q)/(p-1),\frac{sp}{p-1},1\}$ and $u\in W^{1,p}_{\rm loc}(\Omega)\cap L^{p-1}_{s\,p}(\mathbb{R}^N)$ be a weak solution of 
$$
-\Delta_p u+A(-\Delta_p)^s u = f \qquad \mbox{ in }\Omega.
$$
Then $u\in C^\delta_{\rm loc}(\Omega)$ for every $0<\delta<\Theta$. 
\par
More precisely, for every $0<\delta<\Theta$ and every ball $B_{4R}(x_0)\Subset\Omega$ such that $R\in (0,1)$, there exists a constant $C=C(N,s,p,q,\delta)>0$ such that 
$$
[u]_{C^{\delta}(B_{R/8}(x_0))}\leq \frac{C}{R^{\delta}}\,\left(\|u\|_{L^\infty(B_{R}(x_0))}+\mathrm{Tail}_{p-1,s\,p,s\,p}(u;x_0,R)+\Big(R^{p-\frac{N}{q}}\,\|f\|_{L^{q}(B_{R})}\Big)^\frac{1}{p-1}\right).
$$
\end{Theorem}

\begin{Corollary}\label{corc1}
Suppose $2\leq p<\infty$, $0<s<1$ and $sp<(p-1)$. Let $\Omega\subset\mathbb{R}^N$ be an open and bounded set, $0\leq A\leq 1$ and $u\in W^{1,p}_{\rm loc}(\Omega)\cap L^{p-1}_{s\,p}(\mathbb{R}^N)$ be a weak solution of 
\[
-\Delta_p u+A(-\Delta_p)^s u=0\qquad \mbox{ in }\Omega.
\]
Then $u\in C^{1,\alpha}_{\rm loc}(\Omega)$ for some $\alpha\in (0,1)$. 
\par
More precisely, for every ball $B_{2R}(x_0)\Subset\Omega$ with $0<R<1$, there exists a constant $C=C(N,s,p,q,\delta)>0$ such that
$$
[\nabla u]_{C^\alpha(B_{R/8}(x_0))}\leq \frac{C}{R^{1+\alpha}}\,\left(\|u\|_{L^\infty(B_{R}(x_0))}+\mathrm{Tail}_{p-1,s\,p,s\,p}(u;x_0,R)\right),
$$
where $\mathrm{Tail}_{p-1,s\,p,s\,p}(u;x_0,R)$ is defined in \eqref{mtail}.\end{Corollary}

\begin{Remark} The reason for which we have included a constant $A$ in the equation in the above results, is that in the proofs we will consider rescaled solutions. For these, a constant appears in front of the operator $(-\Delta_p)^s$.
\end{Remark}

\subsection{Comments on the results}
We first comment on the sharpness of our results, more specifically Theorem \ref{nonthm}. In general, the results are most likely not sharp. For instance, the results in \cite{DFM} give $C^{1,\alpha}$-regularity for solutions  for all $s\in (0,1)$ and all $p\in (1,\infty)$, under the additional assumption that $u\in W^{s,p}(\R^N)$.

However, our results are almost sharp when $(p-N/q)/(p-1)\leq \frac{sp}{p-1}\leq 1$. Indeed, assume $$(p-N/q)/(p-1)\leq \frac{sp}{p-1}\leq 1$$ and let 

$$u(x)=|x|^{\gamma+\e}, \quad \gamma = (p-N/q)/(p-1)$$
for some $\e>0$. Then 
$$
(-\Delta_p)^s u = f, \quad f(x)=C(s,p,\gamma,\e)|x|^{(\gamma+\e-s)(p-1)-s}
$$
with $f\in L^q_{\text{loc}(\R^N)}$ if and only if $\gamma +\e>(sp-N/q)/(p-1)$. Moreover, 
$$
-\Delta_p u = g, \quad g(x)=C(p,\gamma,\e)|x|^{(\gamma+\e-1)(p-1)-1}
$$
with $g\in L^q_{\text{loc}(\R^N)}$ if and only if $\gamma +\e>(p-N/q)/(p-1)$. It is clear that $u\not\in C^{\alpha}(B_1)$ for any $\alpha>\gamma+\e$. This shows that in this regime of parameters, the results of Theorem \ref{nonthm} are almost sharp.

Now we turn our attention to the H\"older exponents in Theorem \ref{teo:1} and Theorem \ref{nonthm}. Note that even in the case when $f$ is smooth Theorem \ref{nonthm} only gives almost H\"older regularity of order $\min\{sp/(p-1),1\}$, while we for $f=0$ reach almost Lipschitz regularity in Theorem \ref{teo:1}. The reason for this discrepancy is that we prove Theorem \ref{nonthm} by treating the inhomogeneous equation as a perturbation of the homogeneous one. The restriction of the exponent arises when we need a uniform control of the decay at infinity at different scales, see \eqref{eq:spdecay}. It may be possible to treat this as a perturbation of the homogeneous $p$-Laplace equation instead, but we were not able to control the decay at infinity in such an approach.

We also make a small comment regarding the assumption $sp<(p-1)$ in Corollary \ref{corc1}. This assumption arises as a condition for when $(-\Delta_p)^s u$ is bounded for almost Lipschitz functions $u$. The result is then obtained by treating $(-\Delta_p)^s u$ as a bounded term.

\subsection{Known results}
In the homogeneous setting $f=0$ and for $p=2$, equation \eqref{maineqn} reads
\begin{equation}\label{maineqnclas}
-\Delta u+(-\Delta)^s u=0.
\end{equation}
Based on the theory of probability and analysis, equation \eqref{maineqnclas} has been intensely studied in recent years. We mention the work of Foondun \cite{Fo}, where a Harnack inequality and local H\"older continuity are established. We also refer to the Chen-Kim-Song-Vondra\v{c}ek \cite{CKSV}, Chen-Kim-Song \cite{CKSheatest}, Chen-Kim-Song-Vondra\v{c}ek \cite{CKSVGreenest}, Athreya-Ramachandran \cite{AR} and the references therein for related results. For the parabolic problem associated with \eqref{maineqnclas}, Barlow-Bass-Chen-Kassmann \cite{BBCK},  Chen-Kumagai \cite{CK} proved a Harnack inequality and local H\"older continuity.
 
Recently, the regularity theory has also been developed by a purely analytic approach.
For the linear case $p=2$, existence, local boundedness, interior Sobolev regularity and a strong maximum principle, along with other qualitative properties of solutions have been established by Biagi-Dipierro-Valdinoci-Vecchi in \cite{BSVV2}. Local boundedness is also established in Dipierro-Proietti Lippi-Valdinoci \cite{DPV20}. For existence and nonexistence results, we refer to Abatangelo-Cozzi \cite{AC21}. We also refer to Biagi-Dipierro-Valdinoci-Vecchi \cite{BDVVfk21, BSVV1}, Dipierro-Ros-Oton-Serra-Valdinoci \cite{DRXJV20}, Dipierro-Proietti Lippi-Valdinoci \cite{DPV21}, Dipierro-Valdinoci \cite{DVap} and the references therein. 

In the nonlinear setting $p\neq 2$, for $f=0$, regularity results of weak solutions in terms of local boundedness, Harnack estimates, local H\"older continuity and semicontinuity results have been obtained in Garain-Kinnunen \cite{GK}. In \cite{BMV21}, Biagi-Mugnai-Vecchi established boundedness and strong maximum principle in the inhomogeneous case. In the case of a bounded function $f$, Biagi-Dipierro-Valdinoci-Vecchi \cite{En21} has obtained local H\"older continuity for globally bounded solutions and Garain-Ukhlov \cite{GU21} studied existence, uniqueness, local boundedness and further qualitative properties of solutions. Moreover, for more general inhomogeneites, local boundedness is proved in Salort-Vecchi \cite{SV21}. Very recently, H\"older and gradient regularity were proved by De Filippis-Mingione in \cite{DFM}, where a general type of mixed nonlinear problems are considered. Even a mix of different orders and different homogeneities of the operators is allowed. The results therein that applies to \eqref{maineqn} are proved under the global assumption that $u\in W^{s,p}(\R^N)$. Under this assumption, their results contain ours as a special case.

We also seize the opportunity to mention that very recently, the regularity theory for mixed parabolic equations has gained an increasing amount of attention. In the linear case, a weak Harnack inequality is proved for the parabolic analogue of equation \eqref{maineqnclas} in Garain-Kinnunen \cite{GKwh}. For the nonlinear case, see Fang-Shang-Zhang \cite{FSZ1, SZ1} and Garain-Kinnunen \cite{GKp21}. Among other things,  local boundedness and H\"older continuity have been established.
 
Finally, we wish to mention \cite{Domokos}, where a similar approach using difference quotients has been used to obtain improved regularity for quasilinear subelliptic equations in the Heisenberg group.
\subsection{Plan of the paper}

In Section \ref{sec:prel}, we introduce relevant notation and definitions and certain standard result in function spaces. In Section \ref{sec:ex}, we establish existence and uniqueness using standard methods from functional analysis. The core of the paper is mainly in Section \ref{sec:alip}, where we prove almost Lipschitz regularity for the homogeneous equation, using a Moser-type argument that results in an improved differentiability that can be iterated. Here we also prove Corollary \ref{corc1}. This is followed by Section \ref{sec:inhomo}, where the local boundedness and higher H\"older regularity for the inhomogeneous equation is established. This is based on approxmation with the homogenous equation. Finally, in the Appendix, we include a list of pointwise inequalities that are used throughout the paper.

\begin{ack}
 E. L. is supported by the Swedish Research Council, grant no. 2017-03736. Part of this material is based upon work supported by the Swedish Research Council under grant no. 2016-06596 while the second author were participating in the research program ``Geometric Aspects of Nonlinear Partial Differential Equations'', at Institut Mittag-Leffler in Djursholm, Sweden, during the fall of 2022.
 \end{ack}
 
\section{Preliminaries}\label{sec:prel}
In this section, we present some auxiliary results needed in the rest of the paper. Throughout the paper, we shall use the notation that follows. We denote by $B_r(x_0)$, the ball of radius $r$ centered at $x_0$. When $x_0=0$, we will simply write $B_r$. 
It will also be convenient to use the notation $u^+=\max\{u,0\}$. The monotone and $(p-1)$-homogeneous function
$$
J_p(a)=|a|^{p-2}a,\quad a\in\mathbb{R},$$
is expedient when treating equations of $p$-Laplacian type. Discrete differences play an important role. Therefore, for a measurable function $\psi:\mathbb{R}^N\to\mathbb{R}$ and a vector $h\in\mathbb{R}^N$, we define
\[
\psi_h(x)=\psi(x+h),\qquad \delta_h \psi(x)=\psi_h(x)-\psi(x),\qquad \delta^2_h \psi(x)=\delta_h(\delta_h \psi(x))=\psi_{2\,h}(x)+\psi(x)-2\,\psi_h(x).
\]
\subsection{Function spaces}
For  $p\in (1,\infty)$ and $u\in W^{1,p}(\Omega)$, the $W^{1,p}$-seminorm is defined by
$$
[u]^p_{W^{1,p}(\Omega)}:=\int_{\Omega}|\nabla u|^p\,dx.
$$
We also define the critical Sobolev exponent as
\begin{equation}
\label{eq:sobexp}
p^*=\left\{\begin{array}{rcll}
\dfrac{N\,p}{N-p},& \mbox{ if } p<N,\\
&\\
+\infty,& \mbox{ if } p>N,
\end{array}
\right.
\qquad\mbox{ and }\qquad (p^*)'=\left\{\begin{array}{rcll}
\dfrac{N\,p}{N\,p-N+p},& \mbox{ if } p<N,\\
&\\
1,& \mbox{ if } p>N.
\end{array}
\right.
\end{equation}
Moreover, for $0<\delta\leq 1$, we will employ the $\delta$-H\"older seminorm, given by
$$
[u]_{C^\delta(\Omega)}:=\sup_{x\neq y\in \Omega}\frac{|u(x)-u(y)|}{|x-y|^{\delta}}.
$$
For $1\le q<\infty$ and for $0<\beta<2$, we introduce the Besov-type space
\[
\mathcal{B}^{\beta,q}_\infty(\mathbb{R}^N)=\left\{\psi\in L^q(\mathbb{R}^N)\, :\, [\psi]_{\mathcal{B}^{\beta,q}_\infty(\mathbb{R}^N)}<+\infty\right\},	
\]
where 
\[
[\psi]_{\mathcal{B}^{\beta,q}_\infty(\mathbb{R}^N)}:=\sup_{|h|>0} \left\|\frac{\delta_h^2 \psi}{|h|^{\beta}}\right\|_{L^q(\mathbb{R}^N)}.
\]
Similarly, the {\it Sobolev-Slobodecki\u{\i} space} is defined by
\[
W^{\beta,q}(\mathbb{R}^N)=\left\{\psi\in L^q(\mathbb{R}^N)\, :\, [\psi]_{W^{\beta,q}(\mathbb{R}^N)}<+\infty\right\},\qquad 0<\beta<1,
\]
where the seminorm $[\,\cdot\,]_{W^{\beta,q}(\mathbb{R}^N)}$ reads
\[
[\psi]_{W^{\beta,q}(\mathbb{R}^N)}=\left(\iint_{\mathbb{R}^N\times \mathbb{R}^N} \frac{|\psi(x)-\psi(y)|^q}{|x-y|^{N+\beta\,q}}\,dx\,dy\right)^\frac{1}{q}.
\]
These spaces are endowed with their corresponding norms
\[
\|\psi\|_{\mathcal{B}^{\beta,q}_\infty(\mathbb{R}^N)}=\|\psi\|_{L^q(\mathbb{R}^N)}+[\psi]_{\mathcal{B}^{\beta,q}_\infty(\mathbb{R}^N)},
\]
and
\[
\|\psi\|_{W^{\beta,q}(\mathbb{R}^N)}=\|\psi\|_{L^q(\mathbb{R}^N)}+[\psi]_{W^{\beta,q}(\mathbb{R}^N)}.
\]
At times, we will also work with the space $W^{\beta,q}(\Omega)$ for a subset $\Omega\subset \mathbb{R}^N$,
\[
W^{\beta,q}(\Omega)=\left\{\psi\in L^q(\Omega)\, :\, [\psi]_{W^{\beta,q}(\Omega)}<+\infty\right\},\qquad 0<\beta<1,
\]
where we define
\[
 [\psi]_{W^{\beta,q}(\Omega)}=\left(\iint_{\Omega\times \Omega} \frac{|\psi(x)-\psi(y)|^q}{|x-y|^{N+\beta\,q}}\,dx\,dy\right)^\frac{1}{q}.
\]

\subsection{Tail spaces} 
In the study of nonlocal equations, the global behavior of solutions comes into play. This is entailed by the {\it tail space}
\[
L^{q}_{\alpha}(\mathbb{R}^N)=\left\{u\in L^{q}_{\rm loc}(\mathbb{R}^N)\, :\, \int_{\mathbb{R}^N} \frac{|u|^q}{1+|x|^{N+\alpha}}\,dx<+\infty\right\},\qquad q>0 \mbox{ and } \alpha>0,
\]
and measured by the quantity
\begin{equation}\label{mtail}
\mathrm{Tail}_{q,\alpha,\beta}(u;x_0,R)=\left[R^{\beta}\,\int_{\mathbb{R}^N\setminus B_R(x_0)} \frac{|u|^q}{|x-x_0|^{N+\alpha}}\,dx\right]^\frac{1}{q},
\end{equation}
defined for every $x_0\in\mathbb{R}^N$, $R>0,\,\beta>0$ and $u\in L^q_{\alpha}(\mathbb{R}^N)$. We observe that the quantity above is always finite, for a function $u\in L^q_{\alpha}(\mathbb{R}^N)$.

\subsection{Auxiliary results for functions spaces}
The next result asserts that the standard Sobolev space is continuously embedded in the fractional Sobolev space, see \cite[Proposition 2.2]{Hitchhiker'sguide}.
The argument uses \normalcolor the smoothness property of $\Omega$ so that we can extend functions from $W^{1,p}(\Omega)$ to $W^{1,p}(\R^N)$ and that the extension operator is bounded.

\begin{Lemma}\label{locnon}
Let $\Omega$ be a smooth bounded domain in $\mathbb{R}^N$, $1<p<\infty$ and $0<s<1$. 
There exists a positive constant $C=C(N,p,s,\Omega)$ \normalcolor such that
$
\|u\|_{W^{s,p}(\Omega)}\leq C\|u\|_{W^{1,p}(\Omega)}
$
for every $u\in W^{1,p}(\Omega)$.
\end{Lemma}

The following result for the fractional Sobolev spaces with zero boundary value follows from \cite[Lemma 2.1]{Silvaarxiv}.
The main difference compared to Lemma \ref{locnon} is that the result holds for any bounded domain, 
since for the Sobolev spaces with zero boundary value, we may always extend by zero.

\begin{Lemma}\label{locnon1}
Let $\Omega$ be a bounded domain in $\mathbb{R}^N$, $1<p<\infty$ and $0<s<1$. Then there exists a positive constant $C=C(N,p,s,\Omega)$ such that
\[
\int_{\mathbb{R}^n}\int_{\mathbb{R}^n}\frac{|u(x)-u(y)|^p}{|x-y|^{N+ps}}\,dx\, dy
\leq C\int_{\Omega}|\nabla u(x)|^p\,dx
\]
for every $u\in W_0^{1,p}(\Omega)$.
Here we consider the zero extension of $u$ to the complement of $\Omega$.
\end{Lemma}

 The following result is a local version of \cite[Lemma 2.3]{BL}.
\begin{Lemma}\label{2.3} Let $\beta\in (0,1)$, $p\in (1,\infty)$, $x_0\in\mathbb{R}^N$, $R>0$ and $h_1>0$. Suppose 
\begin{equation}\label{uas1}
\begin{split}
u\in L^p(B_{R+\frac{7h_1}{2}}(x_0))\quad\text{and}\quad \sup_{0<|h|<h_1}\left\|\frac{\delta_h^{2}u}{|h|^\beta}\right\|_{L^p(B_{R+\frac{5h_1}{2}}(x_0))}<\infty.
\end{split}
\end{equation}
Then
\begin{equation}\label{2.3eqn}
\begin{split}
\sup_{0<|h|<h_1}\left\|\frac{\delta_h u}{|h|^{\beta}}\right\|_{L^p(B_R(x_0))}&\leq\frac{C}{1-\beta}\left\{\sup_{0<|h|<h_1}\left\|\frac{\delta_h^{2}u}{|h|^{\beta}}\right\|_{L^p(B_{R+\frac{5h_1}{2}}(x_0))}+(h_1^{-\beta}+1)\|u\|_{L^p(B_{R+\frac{7h_1}{2}}(x_0))}\right\}.
\end{split}
\end{equation}
Here $C=C(N,p)>0$.
\end{Lemma}
\begin{proof}
Without loss of generality, we assume that $x_0=0$. Let $0<|h|<h_1$. Let $\eta\in C_c^{\infty}(B_{R+\frac{h_1}{2}})$ be such that $0\leq\eta\leq 1$,\,$|\nabla \eta|\leq\frac{C}{h_1}$,\,$\|D^2 \eta\|\leq\frac{C}{h_1^{2}}$ in $B_{R+\frac{h_1}{2}}$ for some constant $C=C(N,p)>0$ and $\eta\equiv 1$ in $B_R$. Then 
\begin{equation}\label{eta1}
\begin{split}
\|\eta_{2h}\|_{L^\infty(B_{\hat{R}+\frac{5h_1}{2}})}&\leq 1,\quad\|\delta_h (\eta_h)\|_{L^\infty(B_{\hat{R}+\frac{5h_1}{2}})}\leq\frac{C|h|}{h_1},\quad\|\delta_h^{2}\eta\|_{L^\infty(B_{\hat{R}+\frac{5h_1}{2}})}\leq \frac{C|h|^2}{h_1^{2}},
\end{split}
\end{equation}
for some constant $C=C(N,p)>0$. Note that the functions $\eta_{2h},\,\delta_h\eta_h$ and $\delta_h ^2 \eta$ have support inside $B_{R+\frac{5h_1}{2}}$.
Moreover, we obtain
\begin{equation}\label{eta2}
\delta_h^{2}(u\eta)=\eta_{2h}\delta_h^{2}u+2\delta_h u\,\delta_h(\eta_h)+u\delta_h^{2}\eta.
\end{equation}
By the hypothesis \eqref{uas1} and $\eta\in C_c^{\infty}(B_{R+\frac{h_1}{2}})$, it follows that $u\eta\in \mathcal{B}_{\infty}^{\beta,p}(\mathbb{R}^N)$. Then by \cite[Lemma 2.3]{BL}, we have
\begin{equation}\label{BLap1}
\begin{split}
\sup_{0<|h|<h_1}\left\|\frac{\delta_h\,(u\eta)}{|h|^{\beta}}\right\|_{L^p(\mathbb{R}^N)}&\leq \frac{C}{1-\beta}\Big\{\sup_{0<|h|<h_1}\left\|\frac{\delta_h^{2}\,(u\eta)}{|h|^{\beta}}\right\|_{L^p(\mathbb{R}^N)}+(h_1^{-\beta}+1)\|u\eta\|_{L^p(\mathbb{R}^N)}\Big\}.
\end{split}
\end{equation}
Using the above properties of $\eta$, \eqref{eta1}--\eqref{BLap1} and the fact that $0<\beta<1$, we have\normalcolor
\[
\begin{split}
\sup_{0<|h|<h_1}\left\|\frac{\delta_h(u\eta)}{|h|^{\beta}}\right\|_{L^p(B_R)}&\leq \sup_{0<|h|<h_1}\left\|\frac{\delta_h(u\eta)}{|h|^{\beta}}\right\|_{L^p(\mathbb{R}^N)}\\
&\leq\frac{C}{1-\beta}\Big\{\sup_{0<|h|<h_1}\left\|\frac{\delta_h^{2}\,(u\eta)}{|h|^{\beta}}\right\|_{L^p(\mathbb{R}^N)}+(h_1^{-\beta}+1)\|u\eta\|_{L^p(\mathbb{R}^N)}\Big\}\\
&\leq \frac{C}{1-\beta}\left\{\left\|\frac{\delta_h^{2}(u\eta)}{|h|^{\beta}}\right\|_{L^p(\mathbb{R}^N)}+(h_1^{-\beta}+1)\|u\|_{L^p(B_{R+\frac{h_0}{2}})}\right\}\\
&\leq\frac{C}{1-\beta}\Bigg\{\sup_{0<|h|<h_1}\left\|\frac{\delta_h^{2}u}{|h|^{\beta}}\right\|_{L^p(B_{R+\frac{5h_1}{2}})}+\sup_{0<|h|<h_1}\left\|\frac{|h|^{1-\beta}}{h_1}\delta_h u\right\|_{L^p(B_{R+\frac{5h_1}{2}})}\\
&\quad+\sup_{0<|h|<h_1}\left\|\frac{|h|^{2-\beta}}{h_1^{2}}u\right\|_{L^p(B_{R+\frac{5h_1}{2}})}+(h_1^{-\beta}+1)\|u\|_{L^p(B_{R+\frac{h_1}{2}})}\Bigg\}\\
&\leq\frac{C}{1-\beta}\left\{\sup_{0<|h|<h_1}\left\|\frac{\delta_h^{2}u}{|h|^{\beta}}\right\|_{L^p(B_{R+\frac{5h_1}{2}})}+(h_1^{-\beta}+1)\|u\|_{L^p(B_{R+\frac{7h_1}{2}})}\right\},
\end{split}
\]
for some $C=C(N,p)$. This proves the result.
\end{proof}
Our next result is a local version of \cite[Proposition 2.4]{BL}.
\begin{Lemma}\label{2.4}
Let $\alpha\in (1,2)$, $p\in (1,\infty)$, $R>0$, $x_0\in\mathbb{R}^N$ and $h_1>0$.  Suppose 
\begin{equation}\label{uas2}
\begin{split}
u\in L^p(B_{R+6h_1}(x_0))\quad\text{and}\quad \sup_{0<|h|<h_1}\left\|\frac{\delta_h^{2}u}{|h|^\alpha}\right\|_{L^p(B_{R+5h_1}(x_0))}<\infty.
\end{split}
\end{equation}
Then
\begin{equation}\label{alpha}
\begin{split}
\|\nabla u\|_{L^p(B_R(x_0))}&\leq C\Bigg\{\Big(1+\frac{h_1^{-\alpha}+h_1^{-1}}{(\alpha-1)(2-\alpha)}+\frac{h_1^{-\alpha}}{\alpha-1}\Big)\|u\|_{L^p(B_{R+6h_1}(x_0))}\\
&\quad+\frac{3-\alpha}{(\alpha-1)(2-\alpha)}\sup_{0<|h|<h_1}\left\|\frac{\delta_h^{2}u}{|h|^{\alpha}}\right\|_{L^p(B_{R+5h_1}(x_0))}\Bigg\}
\end{split}
\end{equation}
where $C=C(N,p)>0$.
\end{Lemma}
\begin{proof}
Without loss of generality, we assume that $x_0=0$. Let $\eta\in C_c^{\infty}(B_{R+\frac{h_1}{2}})$ be as defined in \eqref{eta1}. Using the assumption \eqref{uas2} and $\eta\in C_c^{\infty}(B_{R+\frac{h_1}{2}})$, we have $u\eta\in \mathcal{B}_{\infty}^{\alpha,p}(\mathbb{R}^N)$. Therefore, by \cite[Propsotion 2.4]{BL}, we get
\begin{equation}\label{BLap2}
\begin{split}
\|\nabla(u\eta)\|_{L^p(\mathbb{R}^N)}&\leq C\|u\eta\|_{L^p(\mathbb{R}^N)}+\frac{C}{\alpha-1}\sup_{|h|>0}\left\|\frac{\delta_h^{2}(u\eta)}{|h|^{\alpha}}\right\|_{L^p(\mathbb{R}^N)},
\end{split}
\end{equation}
for some $C=C(N,p)>0$. Next, using the properties of $\eta$ from \eqref{eta1} and \eqref{eta2}, we observe that
\begin{equation}\label{ob1}
\begin{split}
\sup_{|h|>0}\left\|\frac{\delta_h^{2}(u\eta)}{|h|^{\alpha}}\right\|_{L^p(\mathbb{R}^N)}&=\sup_{|h|>0}\left\|\frac{\eta_{2h}\delta_h^{2}u+2\delta_h u\delta_h(\eta_h)+u\delta_h^{2}\eta}{|h|^{\alpha}}\right\|_{L^p(\mathbb{R}^N)}\\
&\leq C\sup_{0<|h|<h_1}\Bigg\{\left\|\eta_{2h}\frac{\delta_h^{2}u}{|h|^{\alpha}}\right\|_{L^p(B_{R+\frac{5h_1}{2}})}+\left\|\delta_{h}(\eta_h)\frac{\delta_h\,u}{|h|^{\alpha}}\right\|_{L^p(B_{R+\frac{5h_1}{2}})}\\
&\quad+\left\|\delta_{h}^2(\eta)\frac{u}{|h|^{\alpha}}\right\|_{L^p(B_{R+\frac{5h_1}{2}})}\Bigg\}\\
&\leq C\sup_{0<|h|<h_1}\Bigg\{
\left\|\frac{\delta_h^{2}u}{|h|^{\alpha}}\right\|_{L^p(B_{R+\frac{5h_1}{2}})}+\frac{1}{h_1}\left\|\frac{\delta_h\,u}{|h|^{\alpha-1}}\right\|_{L^p(B_{R+\frac{5h_1}{2}})}+h_1^{-\alpha}\|u\|_{L^p(B_{R+\frac{5h_1}{2}})}\Bigg\}, 
\end{split}
\end{equation}
for some positive constant $C=C(N,p)>0$. Now we estimate the second integral in the RHS of \eqref{ob1}. To this end, using \eqref{uas2}, we get
\begin{equation}\label{uas3}
\begin{split}
u\in L^p(B_{R+6h_1})\quad\text{and}\quad\sup_{0<|h|<h_1}\left\|\frac{\delta_h^{2}\,u}{|h|^{\alpha-1}}\right\|_{L^p(B_{R+5h_1})}<\infty.
\end{split}
\end{equation}
Since $0<\alpha-1<1$, by Lemma \ref{2.3}, it follows that
\begin{equation}\label{2.3ap}
\begin{split}
\sup_{0<|h|<h_1}\left\|\frac{\delta_h\,u}{|h|^{\alpha-1}}\right\|_{L^p(B_{R+\frac{5h_1}{2}})}&\leq \frac{C}{2-\alpha}\Bigg\{\sup_{0<|h|<h_1}\left\|\frac{\delta_h^{2}\,u}{|h|^{\alpha-1}}\right\|_{L^p(B_{R+5h_1})}+(h_1^{-\alpha-1}+1)\|u\|_{L^p(B_{R+6h_1})}\Bigg\}\\
&\leq\frac{C}{2-\alpha}\Bigg\{h_1\sup_{0<|h|<h_1}\left\|\frac{\delta_h^{2}\,u}{|h|^{\alpha}}\right\|_{L^p(B_{R+5h_1})}+(h_1^{-\alpha-1}+1)\|u\|_{L^p(B_{R+6h_1})}\Bigg\},
\end{split}
\end{equation}
for some $C=C(N,p)$. Combining the estimates \eqref{2.3ap} and \eqref{ob1} in \eqref{BLap2} and noting that $\eta\equiv 1$ in $B_{R}$, the result follows.
\end{proof}
\begin{Lemma} \label{lem:discdiffW1p} Suppose $u\in W^{1,p}(\R^N)$, where $p\in (1,\infty)$. Then
$$
\sup_{|h|>0}\Big\|\frac{\delta_h u}{h}\Big\|_{L^p(\R^N)}\leq \|\nabla u\|_{L^p(\R^N)}.
$$
\end{Lemma}

\begin{proof} We have
$$
u(x+h)-u(x) = \int_0^1\nabla u (x+th)\cdot h  dt.
$$
Therefore, by H\"older's inequality
$$
\Big|\frac{u(x+h)-u(x)}{|h|}\Big|^p\leq \int_0^1|\nabla u (x+th)|^p  dt.
$$
Upon integrating, we obtain 
\[
\begin{split}
\int_{\R^N}\Big|\frac{u(x+h)-u(x)}{|h|}\Big|^p dx&\leq \int_{\R^N}\int_0^1|\nabla u (x+th)|^p  dt dx \\
&\leq \int_0^1\int_{\R^N}|\nabla u (x+th)|^p  dx dt\\
& \leq \|\nabla u\|_{L^p(\R^N)}^p.
\end{split}
\]
\end{proof}
We seize the opportunity to mention that a local version of the above lemma can be found in Theorem 3 on page 277 in \cite{Evans}.
\subsection{Weak solutions}
Below, we define weak solutions of \eqref{maineqn}, allowing also for a factor $A$ that will be needed in the sequel, when treating rescaled solutions.
\begin{Definition}\label{subsupsolution}
Let $1<p<\infty,\,0<s<1$ and $A\geq 0$. Suppose $f\in L^q(\Omega)$, with
\[
q\ge (p^*)'\quad \mbox{ if } p\not =N\qquad \mbox{ or }\qquad q>1\quad \mbox{ if }p=N.
\] We say that $u\in W_{\loc}^{1,p}(\Omega)\cap L^{p-1}_{sp}(\mathbb{R}^N)$ is a weak subsolution (or supersolution) of 
$$
-\Delta_p u+A(-\Delta_p)^{s}u=f\text{ in }\Omega,
$$
if for every $K\Subset\Omega$ and for every nonnegative $\phi\in W_{0}^{1,p}(K)$, we have
\begin{equation}\label{wksol}
\begin{gathered}
\int_{K}|\nabla u|^{p-2}\nabla u\cdot\nabla\phi\,dx+A\int_{\mathbb{R}^N}\int_{\mathbb{R}^N}J_p((u(x)-u(y)){(\phi(x)-\phi(y))}\,d\mu\leq\int_{K}f\phi\,dx\quad (\text{or }\geq),
\end{gathered}
\end{equation}
where
\begin{equation}\label{jpmu}
J_p(a)=|a|^{p-2}a,\quad a\in\mathbb{R},\quad d\mu=|x-y|^{-N-sp}\,dx\,dy.
\end{equation}
We say that $u$ is a weak solution of \eqref{maineqn}, if equality holds in \eqref{wksol} for every $\phi\in W_0^{1,p}(K)$. 
\normalcolor 
\end{Definition}

\begin{Remark}\label{defrmk2}
By Lemma \ref{locnon} and Lemma \ref{locnon1}, Definition \ref{wksol} makes sense.
\end{Remark}

We now  detail the notion of weak solutions to the Dirichlet boundary value problem. For that purpose, given $\Omega\subset\mathbb{R}^N$ an open and bounded set, consider a bounded domain $\Omega^{'}$ such that $\Omega\Subset\Omega^{'}\subset\mathbb{R}^N$\normalcolor. Then for $g\in W^{1,p}(\Omega')$, we define 
\begin{equation}\label{DS}
W_g^{1,p}(\Omega)=\{v\in W^{1,p}(\Omega)\cap L^{p-1}_{sp}(\mathbb{R}^N)\, :\, v-g\in W_0^{1,p}(\Omega)\}.
\end{equation}
When $u\in W_g^{1,p}(\Omega)$ we will repeatedly identify $u$ as being extended by $g$ outside of $\Omega$.

\begin{Definition}[Dirichlet problem]
Let $1<p<\infty,\,0<s<1$ and $A\geq 0$. Suppose $\Omega\Subset \Omega'\subset\mathbb{R}^N$ be two open and bounded sets, $f\in L^q(\Omega)$, with
\[
q\ge (p^*)'\quad \mbox{ if } p\not =N\qquad \mbox{ or }\qquad q>1\quad \mbox{ if }p=N,
\]
and $g\in W^{1,p}(\Omega')\cap L^{p-1}_{sp}(\mathbb{R}^N)$. We say that $u\in W_g^{1,p}(\Omega)$ is a {\it weak solution} of the boundary value problem
\begin{equation}
\label{BVP}
\left\{\begin{array}{rcll}
-\Delta_p u+A(-\Delta_p)^s\,u&=&f,&\mbox{ in }\Omega,\\
u&=&g,& \mbox{ in }\mathbb{R}^N\setminus \Omega,
\end{array}
\right.
\end{equation}
if for every $\phi\in W_0^{1,p}(\Omega)$, we have
\begin{equation}\label{Dpwksol}
\begin{gathered}
\int_{\Omega}|\nabla u|^{p-2}\nabla u\cdot\nabla\phi\,dx+A\int_{\mathbb{R}^N}\int_{\mathbb{R}^N}J_p((u(x)-u(y)){(\phi(x)-\phi(y))}\,d\mu=\int_{\Omega}f\phi\,dx,
\end{gathered}
\end{equation}
where $J_p$ and $d\mu$ are defined in \eqref{jpmu} above.
\end {Definition}
\begin{Remark}\label{defrmk1}
Note that Definition \ref{Dpwksol} makes sense by Lemma \ref{locnon} and Lemma \ref{locnon1}, since we may choose a smooth set $K$ such that $\Omega\Subset K\Subset\Omega'$.
\end{Remark}

\section{Existence and uniqueness}\label{sec:ex}
Here we prove existence and uniqueness of solutions of the Dirichlet problem \eqref{BVP}.
\begin{proof}[Proof of Theorem \ref{prop:death!}]
In what follows, whenever $X$ is a normed vector space, we denote by $X^*$ its topological dual.

We first note that $W_0^{1,p}(\Omega)$ is a separable reflexive Banach space. 
We now introduce the operator $\mathcal{A}:W_g^{1,p}(\Omega)\to (W_0^{1,p}(\Omega))^*$ defined by
\[
\begin{split}
\langle\mathcal{A}(v),\varphi\rangle&=\int_{\Omega'}|\nabla v|^{p-2}\nabla v\nabla\phi\,dx+A\iint_{\Omega'\times\Omega'} {J_p(v(x)-v(y))\,\big(\varphi(x)-\varphi(y)\big)}\,d\mu\\
&+2A\,\iint_{\Omega\times (\mathbb{R}^N\setminus\Omega^{'})}{J_p(v(x)-g(y))\,\varphi(x)}\,d\mu,\qquad v\in W_g^{1,p}(\Omega),\ \varphi\in W_0^{1,p}(\Omega),
\end{split}
\]
where $\langle \cdot,\cdot\rangle$ denotes the relevant duality product. We observe that $\mathcal{A}(v)\in (W_0^{1,p}(\Omega))^*$ for every $v\in W_g^{1,p}(\Omega)$ (by Lemma \ref{locnon} and \cite[Remark 1]{KKP}). Moreover, as in the proof of \cite[Lemma 3]{KKP}, we have that $\mathcal{A}$ has the following properties:
\begin{enumerate}
\item for every $v,u\in W_g^{1,p}(\Omega)$, we have
\[
\langle \mathcal{A}(u)-\mathcal{A}(v),u-v\rangle\geq 0,
\]
with equality if and only if $u=v$; This follows from applying Lemma \ref{lem:pineq} to the nonlocal part and noting that for the local term we have the following inequalities (see \cite[Page 11]{Simon}):
\begin{equation}
\label{strictmono}
\langle\mathcal{A}(u)-\mathcal{A}(v),u-v\rangle\geq
\begin{cases}
\displaystyle C_1\Big(\int_{\Omega}|\nabla(u-v)|^p\,dx\Big)^\frac{1}{p},\text{ if }p\geq 2,\\
\frac{\displaystyle C_2\big(\int_{\Omega}|\nabla(u-v)|^p\,dx\big)^\frac{2}{p}}{\displaystyle\left(\left(\int_{\Omega}|\nabla u|^p\,dx\right)^\frac{1}{p}+\left(\int_{\Omega}|\nabla v|^p\,dx\right)^\frac{1}{p}\right)^{2-p}},\text{ if }1<p<2,
\end{cases}
\end{equation}
for some positive constants $C_1,\,C_2$.
\item if $\{u_n\}_{n\in\mathbb{N}}\subset W_g^{1,p}(\Omega)$ converges in $W^{1,p}(\Omega)$ to $u\in W_g^{1,p}(\Omega)$, then
\[
\lim_{n\to\infty} \langle \mathcal{A}(u_n)-\mathcal{A}(u),v\rangle=0\quad \text{for all }v\in W_0^{1,p}(\Omega);
\]
This follows from the application of Lemma \ref{locnon} together with H\"older's inequality and the coupling of weak and strong convergence.
\item From \eqref{strictmono}, it follows that
\[
\lim_{\|u\|_{W^{1,p}(\Omega)}\to+\infty} \frac{\langle \mathcal{A}(u)-\mathcal{A}(g),u-g\rangle}{\|u-g\|_{W^{1,p}(\Omega)}}=+\infty.
\]
\end{enumerate} 
Finally, we introduce the modified functional
\[
\mathcal{A}_0(v):=\mathcal{A}(v+g),\qquad \mbox{ for every } v\in W_0^{1,p}(\Omega).
\]
We observe that $\mathcal{A}_0:W_0^{1,p}(\Omega)\to (W_0^{1,p}(\Omega))^*$.
Moreover, properties (1), (2) and (3) above imply that $\mathcal{A}_0$ is monotone, coercive and hemicontinuous (see \cite[Chapter II, Section 2]{Sh} for the relevant definitions). It is only left to observe that under the standing assumptions, the linear functional 
\[
T_f:v\mapsto \int_\Omega f\,v\,dx,\qquad v\in W_0^{1,p}(\Omega),
\] 
belongs to the topological dual of $W_0^{1,p}(\Omega)$. Notice that for every $v\in W_0^{1,p}(\Omega)$ we have\footnote{We assume for simplicity that $1<p<N$. The cases $p\geq N$ can be treated in the same manner, we leave the details to the reader.}
\[
|T_f(v)|=\left|\int_\Omega f\, v\,dx\right|\le \|f\|_{L^q(\Omega)}\,\|v\|_{L^{q'}(\Omega)}\le |\Omega|^{\frac{1}{q'}-\frac{1}{p^*}}\,\|f\|_{L^q(\Omega)}\,\|v\|_{L^{p^*}(\Omega)},
\]
and the last term can be controlled using the Sobolev embedding $W^{1,p}(\mathbb{R}^N)\to L^{p^*}(\mathbb{R}^N)$ (see \cite{Evans}).
Then by \cite[Corollary 2.2]{Sh}, we obtain the existence of $v\in W_0^{1,p}(\Omega)$ such that
\[
\langle \mathcal{A}_0(v),\varphi\rangle=\langle T_f,\varphi\rangle,\qquad \mbox{ for every }\varphi \in W_0^{1,p}(\Omega).
\]
By definition, this is equivalent to 
\[
\langle \mathcal{A}(v+g),\varphi\rangle=\langle T_f,\varphi\rangle,\qquad \mbox{ for every }\varphi \in W_0^{1,p}(\Omega),
\]
i.e.
\[
\begin{split}
&\int_{\Omega'}|\nabla(v+g)|^{p-2}\nabla(v+g)\nabla\phi\,dx+A\iint_{\Omega'\times\Omega'} {J_p(v(x)+g(x)-v(y)-g(y))\,\big(\varphi(x)-\varphi(y)\big)}\,d\mu\\
&+2A\,\iint_{\Omega\times (\mathbb{R}^N\setminus\Omega')}{J_p(v(x)+g(x)-g(y))\,\varphi(x)}\,d\mu=\int_\Omega f\,\varphi\,dx,
\end{split}
\]
which is the same as \eqref{Dpwksol}, since $v=0$ in $\mathbb{R}^N\setminus\Omega$ and that
\[
\begin{split}
2\,\iint_{\Omega\times (\mathbb{R}^N\setminus\Omega')}&{J_p(v(x)+g(x)-g(y))\,\varphi(x)}\,d\mu\\
&=\iint_{\Omega\times (\mathbb{R}^N\setminus\Omega')}{J_p(v(x)+g(x)-v(y)-g(y))\,\varphi(x)}\,d\mu\\
&-\iint_{(\mathbb{R}^N\setminus\Omega')\times \Omega}{J_p(v(x)+g(x)-v(y)-g(y))\,\varphi(y)}\,d\mu.
\end{split}
\]
Then $v+g$ is the desired solution. Uniqueness now follows from the {\it strict} monotonicity of the operator $\mathcal{A}_0$.
\end{proof}
\begin{Remark}[Variational solutions]
Under the slightly stronger assumption $g\in W^{1,p}(\Omega')\cap L^p_{s\,p}(\mathbb{R}^N)$, existence of the solution to \eqref{BVP} can be obtained by solving the following strictly convex variational problem
\[
\min\left\{\mathcal{F}(v)\, :\, v\in W^{1,p}_g(\Omega)\cap L^p_{s\,p}(\mathbb{R}^N)\right\},
\]
where the functional $\mathcal{F}$ is defined by
\[
\begin{split}
\mathcal{F}(v)&=\frac{1}{p}\int_{\Omega}|\nabla v|^p\,dx+\frac{A}{p}\,\iint_{\Omega'\times \Omega'} {|v(x)-v(y)|^p}\,d\mu+\frac{2A}{p}\,\iint_{\Omega\times (\mathbb{R}^N\setminus\Omega')}{|v(x)-g(y)|^p}\,d\mu-\int_\Omega f\,v\,dx.
\end{split}
\]\normalcolor
Existence of a minimizer can be obtained using the Direct Methods in the Calculus of Variations.
\end{Remark}

\section{Almost Lipschitz regularity for the homogeneous equation}\label{sec:alip}
In this section, we prove the almost Lipschitz regularity for the homogeneous equation. We first start with the result below, where we differentiate the equation discretely and test with powers of $\delta_h u$. This yields an iteration scheme of Moser-type. This is the core of the paper.
\begin{prop}\label{prop1}
Let $2\leq p<\infty$, $0<s<1$ and $0\leq A\leq 1$. Suppose that $u\in W^{1,p}_{\mathrm{loc}}(B_2(x_0))\cap L^{p-1}_{sp}(\mathbb{R}^N)$ is a weak solution of $-\Delta_p u+A(-\Delta_p)^s u=0$ in $B_2(x_0)$. Assume that
\begin{equation}\label{as1}
\|u\|_{L^\infty(B_1(x_0))}\leq 1,\qquad \int_{\mathbb{R}^N\setminus B_1(x_0)}\frac{|u(y)|^{p-1}}{|y|^{N+sp}}\,dy\leq 1.
\end{equation}
Let $0<h_0<\frac{1}{10}$ and $R$ be such that  $4h_0<R\leq 1-5h_0$ and $\nabla u\in L^q(B_{R+4h_0}(x_0))$ for some $q\geq p$.
Then
\begin{equation}\label{p1est}
\begin{split}
\sup_{0<|h|<h_0}\left\|\frac{\delta_h ^2 u}{|h|^{1+\frac{1}{q+1}}}\right\|^{q+1}_{L^{q+1}(B_{R-4h_0}(x_0))}&\leq C(1+A)\left(\int_{B_{R+4h_0}(x_0)}|\nabla u|^q\,dx+1\right),
\end{split}
\end{equation}
for some constant $C=C(N,h_0,p,q,s)>0$.
\end{prop}
\begin{proof}
Without loss of generality, we assume that $x_0=0$. We divide the proof into five steps. \\
\noindent
\textbf{Step 1: Discrete differentiation of the equation.} Let $r=R-4h_0$ and $\phi\in W^{1,p}(B_R)$ vanish outside $B_{\frac{R+r}{2}}$. Since $u$ is a weak solution of $-\Delta_p u+A(-\Delta_p)^s u=0$ in $B_2$, from Definition \ref{subsupsolution}, we have
\begin{equation}\label{defap1}
\int_{B_R}|\nabla u|^{p-2}\nabla u\nabla\phi\,dx+A\int_{\mathbb{R}^n}\int_{\mathbb{R}^n}(J_p(u(x)-u(y))){(\phi(x)-\phi(y))}\,d\mu=0.
\end{equation}
Let $h\in\mathbb{R}^n\setminus\{0\}$ be such that $|h|<h_0$. Choosing $\phi=\phi_{-h}$ in \eqref{defap1} and using a change of variables, we have
\begin{equation}\label{tst1}
\int_{B_R}|\nabla u_h|^{p-2}\nabla u_h\nabla\phi\,dx+A\int_{\mathbb{R}^N}\int_{\mathbb{R}^N}(J_p(u_h(x)-u_h(y))){(\phi(x)-\phi(y))}\,d\mu=0.
\end{equation}
Subtracting \eqref{defap1} with \eqref{tst1} and dividing the resulting equation by $|h|$, we obtain
\begin{equation}\label{tst2}
\begin{split}
&\int_{B_R}\frac{(|\nabla u_h|^{p-2}\nabla u_h-|\nabla u|^{p-2}\nabla u)}{|h|}\nabla\phi\,dx\\
&+A\int_{\mathbb{R}^N}\int_{\mathbb{R}^N}\frac{(J_p(u_h(x)-u_h(y))-(J_p(u(x)-u(y)))}{|h|}{(\phi(x)-\phi(y))}\,d\mu=0,
\end{split}
\end{equation}
for every $\phi\in W^{1,p}(B_R)$ vanishing outside $B_{\frac{R+r}{2}}$. Let $\eta$ be a nonnegative Lipschitz cut-off function such that
$$
\eta\equiv 1\text{ on } B_r,\quad \eta\equiv 0\text{ on }\mathbb{R}^N\setminus B_{\frac{R+r}{2}},\quad |\nabla\eta|\leq\frac{C}{R-r}=\frac{C}{4h_0},
$$
for some constant $C=C(N)>0$.
Suppose $\alpha\geq 1$, $\theta>0$ and testing \eqref{tst2} with
$$
\phi=J_{\alpha+1}\Big(\frac{u_h-u}{|h|^{\theta}}\Big)\eta^p,\quad 0<|h|<h_0,
$$
we get
\begin{equation}\label{tst3}
I+AJ=0,
\end{equation}
where
\begin{equation}\label{l}
\begin{split}
I&=\int_{B_R}\frac{(|\nabla u_h|^{p-2}\nabla u_h-|\nabla u|^{p-2}\nabla u)}{|h|^{1+\theta\alpha}}\nabla(J_{\alpha+1}(u_h-u)\eta^p)\,dx
\end{split}
\end{equation}
and
\begin{equation}\label{nl}
\begin{split}
J&=\int_{\mathbb{R}^n}\int_{\mathbb{R}^n}\frac{(J_p(u_h(x)-u_h(y))-(J_p(u(x)-u(y)))}{|h|^{1+\theta\alpha}}\\
&\times(J_{\alpha+1}(u_h(x)-u(x))\eta^p(x)-J_{\alpha+1}(u_h(y)-u(y))\eta^p(y))\,d\mu.
\end{split}
\end{equation}
\textbf{Step 2: Estimate of the local integral $I$.} We observe that
\begin{equation}\label{lI}
\begin{split}
I_{12}&=(|\nabla u_h|^{p-2}\nabla u_h-|\nabla u|^{p-2}\nabla u)\nabla(J_{\alpha+1}({u_h-u})\eta^p)\\
&=(|\nabla u_h|^{p-2}\nabla u_h-|\nabla u|^{p-2}\nabla u)\eta^p\nabla(J_{\alpha+1}(u_h-u))\\
&\quad+(|\nabla u_h|^{p-2}\nabla u_h-|\nabla u|^{p-2}\nabla u)J_{\alpha+1}(u_h-u)\nabla(\eta^p)\\
&\geq(|\nabla u_h|^{p-2}\nabla u_h-|\nabla u|^{p-2}\nabla u)\eta^p\nabla(J_{\alpha+1}(u_h-u))\\
&\quad-\big||\nabla u_h|^{p-2}\nabla u_h-|\nabla u|^{p-2}\nabla u|\big||u_h-u|^{\alpha}|\nabla(\eta^p)|\\
&:=I_1-I_2.
\end{split}
\end{equation}
\textbf{Estimate of $I_1$:} Since $p\geq 2$, using Lemma \ref{in1} and that $\alpha\geq 1$, we get
\begin{equation}\label{lI1}
\begin{split}
I_1&=(|\nabla u_h|^{p-2}\nabla u_h-|\nabla u|^{p-2}\nabla u)\eta^p\nabla(J_{\alpha+1}(u_h-u))\\
&=\alpha(|\nabla u_h|^{p-2}\nabla u_h-|\nabla u|^{p-2}\nabla u)\nabla(u_h-u)|u_h-u|^{\alpha-1}\eta^p\\
&\geq \frac{4}{p^2}(|\nabla u_h|^\frac{p-2}{2}\nabla u_h-|\nabla u|^\frac{p-2}{2}\nabla u\big|^2 |u_h-u|^{\alpha-1}\eta^{p}.
\end{split}
\end{equation}
Moreover, for $p\geq 2$, using Lemma \ref{lem:pineq}, we have
\begin{equation}\label{lI1est}
\begin{split}
I_1&=\alpha(|\nabla u_h|^{p-2}\nabla u_h-|\nabla u|^{p-2}\nabla u)\nabla(u_h-u)|u_h-u|^{\alpha-1}\eta^p\\
&\geq p2^{2-p}|\nabla(u_h-u)|^p|u_h-u|^{\alpha-1}\eta^p\\
&=p2^{2-p}\Big(\frac{p}{\alpha+p-1}\Big)^p\Big|\nabla\Big(|u_h-u|^\frac{\alpha-1}{p}(u_h-u)\Big)\Big|^p\eta^p\\
&\geq p2^{2-p}\Big(\frac{p}{\alpha+p-1}\Big)^p\Big\{2^{-p}\Big|\nabla\Big(|u_h-u|^\frac{\alpha-1}{p}(u_h-u)\eta\Big)\Big|^p-\Big||u_h-u|^\frac{\alpha-1}{p}(u_h-u)\Big)\Big|^p|\nabla\eta|^p\Big\}.
\end{split}
\end{equation}
\noindent
\textbf{Estimate of $I_2$:} Since $p\geq 2$, using Lemma \ref{in2} and Young's inequality with exponents 2 and 2, we obtain
\begin{equation}\label{lI2}
\begin{split}
I_2&=\big||\nabla u_h|^{p-2}\nabla u_h-|\nabla u|^{p-2}\nabla u|\big||u_h-u|^{\alpha}|\nabla(\eta^p)|\\
&\leq (p-1)(|\nabla u_h|^\frac{p-2}{2}+|\nabla u|^\frac{p-2}{2})\big||\nabla u_h|^\frac{p-2}{2}\nabla u_h-|\nabla u|^\frac{p-2}{2}\nabla u\big||u_h-u|^{\alpha}2\eta^\frac{p}{2}|\nabla(\eta^\frac{p}{2})|\\
&=\Big((p-1)(|\nabla u_h|^\frac{p-2}{2}+|\nabla u|^\frac{p-2}{2})\big||u_h-u|^{\frac{\alpha+1}{2}}|\nabla(\eta^\frac{p}{2})|\Big)  \Big(|\nabla u_h|^\frac{p-2}{2}\nabla u_h-|\nabla u|^\frac{p-2}{2}\nabla u\big||u_h-u|^{\frac{\alpha-1}{2}}2\eta^\frac{p}{2}\Big)\\
&\leq C(p,\epsilon)\big(|\nabla u_h|^\frac{p-2}{2}+|\nabla u|^\frac{p-2}{2}\big)^2\big||u_h-u|^{\alpha+1}|\nabla(\eta^\frac{p}{2})|^2\\
&\quad+\epsilon(|\nabla u_h|^\frac{p-2}{2}\nabla u_h-|\nabla u|^\frac{p-2}{2}\nabla u\big|^2 |u_h-u|^{\alpha-1}\eta^{p}\\
&\leq C(p,\epsilon)\big(|\nabla u_h|^\frac{p-2}{2}+|\nabla u|^\frac{p-2}{2}\big)^2\big||u_h-u|^{\alpha+1}|\nabla(\eta^\frac{p}{2})|^2+\frac{\epsilon p^2}{4}I_1,
\end{split}
\end{equation}
for some $\epsilon\in(0,\frac{4}{p^2})$, where to obtain the last inequality above, we have used the estimate \eqref{lI1}. Thus, using the estimate \eqref{lI2} in \eqref{lI}, it follows that
\begin{equation}\label{lI12}
\begin{split}
I_{12}&\geq cI_1-C\big(|\nabla u_h|^\frac{p-2}{2}+|\nabla u|^\frac{p-2}{2}\big)^2\big||u_h-u|^{\alpha+1}|\nabla(\eta^\frac{p}{2})|^2\\
&\geq cp2^{2-p}\Big(\frac{p}{\alpha+p-1}\Big)^p\Big\{2^{-p}\Big|\nabla\Big(|u_h-u|^\frac{\alpha-1}{p}(u_h-u)\eta\Big)\Big|^p-\Big||u_h-u|^\frac{\alpha-1}{p}(u_h-u)\Big)\Big|^p|\nabla\eta|^p\Big\}\\
&\quad-C\big(|\nabla u_h|^\frac{p-2}{2}+|\nabla u|^\frac{p-2}{2}\big)^2\big||u_h-u|^{\alpha+1}|\nabla(\eta^\frac{p}{2})|^2,
\end{split}
\end{equation}
for some positive constants $c,\,C$ depending on $p$. Therefore, using the estimate \eqref{lI12} in \eqref{l}, we have
\begin{equation}\label{lIest}
\begin{split}
I&=\int_{B_R}\frac{I_{12}}{|h|^{1+\theta\alpha}}\,dx\\
&\geq c\int_{B_R}\Big|\nabla\Big(\frac{|u_h-u|^\frac{\alpha-1}{p}(u_h-u)\eta}{|h|^\frac{1+\theta\alpha}{p}}\Big)\Big|^p\,dx-c\int_{B_R}\frac{\Big||u_h-u|^\frac{\alpha-1}{p}(u_h-u)\Big|^p|\nabla\eta|^p}{|h|^{1+\theta\alpha}}\,dx\\
&\quad-C\int_{B_R}\frac{\big(|\nabla u_h|^\frac{p-2}{2}+|\nabla u|^\frac{p-2}{2}\big)^2\big||u_h-u|^{\alpha+1}|\nabla(\eta^\frac{p}{2})|^2}{|h|^{1+\theta\alpha}}\,dx\\
&:=cI_{13}
-cI_{14}-CI_{15},
\end{split}
\end{equation}
for some positive constants $c,\,C$ depending on $p,\,\alpha$.\\
\textbf{Estimate of $I_{14}:$} Let $p>2$, then using the properties of $\eta$ and Young's inequality with exponents $\frac{q}{p-2}$ and $\frac{q}{q-p+2}$, using that $\|u\|_{L^\infty(B_1)}\leq 1$ from \eqref{as1}, we have
\begin{equation}\label{lI14}
\begin{split}
I_{14}&=\int_{B_R}\frac{\Big||u_h-u|^\frac{\alpha-1}{p}(u_h-u)\Big|^p|\nabla\eta|^p}{|h|^{1+\theta\alpha}}\,dx\\
&\leq\int_{B_R}\frac{|\delta_h u|^{\alpha+p-1}}{|h|^{1+\theta\alpha}}|\nabla\eta|^p\,dx\\
&\leq\Big(\frac{C}{4h_0}\Big)^p\Big(\int_{B_R}\frac{|\delta_h u|^\frac{\alpha q}{q-p+2}}{|h|^\frac{(1+\theta\alpha)q}{q-p+2}}\,dx+\int_{B_R}|\delta_h u|^\frac{(p-1)q}{p-2}\,dx\Big)\\
&\leq C\Big(\int_{B_R}\frac{|\delta_h u|^\frac{\alpha q}{q-p+2}}{|h|^\frac{(1+\theta\alpha)q}{q-p+2}}\,dx+1\Big),
\end{split}
\end{equation}
for some constant $C=C(N,h_0,p,q)>0$. Note that when $p=2$, again using that $\|u\|_{L^\infty(B_1)}\leq 1$ from \eqref{as1}, we have
$$
|\delta_h u|^{\alpha+1}\leq 2\|u\|_{L^{\infty}(B_{R+h_0})}|\delta_h u|^{\alpha}\leq 2|\delta_h u|^\alpha,
$$
which gives the estimate \eqref{lI14} for $p=2$.
\\
\textbf{Estimate of $I_{15}$:} We observe that
\begin{equation}\label{lI15}
\begin{split}
I_{15}&=\int_{B_R}\frac{\big(|\nabla u_h|^\frac{p-2}{2}+|\nabla u|^\frac{p-2}{2}\big)^2\big||u_h-u|^{\alpha+1}|\nabla(\eta^\frac{p}{2})|^2}{|h|^{1+\theta\alpha}}\,dx\\
&\leq 4\int_{B_R}\frac{\big(|\nabla u_h|^{p-2}+|\nabla u|^{p-2}\big)\big||u_h-u|^{\alpha+1}|\nabla(\eta^\frac{p}{2})|^2}{|h|^{1+\theta\alpha}}\,dx\\
&=4\int_{B_R}\frac{|\nabla u_h|^{p-2}\big||u_h-u|^{\alpha+1}|\nabla(\eta^\frac{p}{2})|^2}{|h|^{1+\theta\alpha}}\,dx+4\int_{B_R}\frac{|\nabla u|^{p-2}\big||u_h-u|^{\alpha+1}|\nabla(\eta^\frac{p}{2})|^2}{|h|^{1+\theta\alpha}}\,dx\\
&:=4(I_{16}+I_{17}).
\end{split}
\end{equation}
\\
\textbf{Estimates of $I_{16}$ and $I_{17}$:} If $p=2$, using the boundedness assumption $\|u\|_{L^\infty(B_1)}\leq 1$ from \eqref{as1} and the properties of $\eta$, we have
\begin{equation}\label{lI16}
\begin{split}
\int_{B_R}\frac{|u_h-u|^{\alpha+1}|\nabla\eta|^2}{|h|^{1+\theta\alpha}}\,dx&\leq \|u\|_{L^{\infty}(B_{R+h_0})}\int_{B_R}\frac{|\nabla\eta|^2|\delta_h u|^{\alpha}}{|h|^{1+\theta\alpha}}\,dx\\
&\leq \Big(\frac{C}{4h_0}\Big)^2\int_{B_R}\frac{|\delta_h u|^{\alpha}}{|h|^{1+\theta\alpha}}\,dx,
\end{split}
\end{equation}
for some $C=C(N,p)>0$. For $p>2$, using Young's inequality with exponents $\frac{q}{p-2}$ and $\frac{q}{q-p+2}$, we get
\begin{equation}\label{lI16est}
\begin{split}
\int_{B_R}\frac{|\nabla u_h|^{p-2}|u_h-u|^{\alpha+1}|\nabla\eta|^2}{|h|^{1+\theta\alpha}}\,dx&\leq C\int_{B_R}|\nabla u_h|^q\,dx+\Big(\frac{C}{h_0}\Big)^\frac{2q}{q-p+2}\int_{B_R}\frac{|\delta_h u|^\frac{(\alpha+1)q}{q-p+2}}{|h|^\frac{(1+\theta\alpha)q}{q-p+2}}\,dx\\
&\leq C\int_{B_R}|\nabla u_h|^q\,dx+C\int_{B_R}\|u\|_{L^\infty(B_{R+h_0})}^{\frac{q}{q-p+2}}\frac{|\delta_h u|^\frac{\alpha q}{q-p+2}}{|h|^\frac{(1+\theta\alpha)q}{q-p+2}}\,dx\\
&\leq C\int_{B_R}|\nabla u_h|^q\,dx+C\int_{B_R}\frac{|\delta_h u|^\frac{\alpha q}{q-p+2}}{|h|^\frac{(1+\theta\alpha)q}{q-p+2}}\,dx 
\end{split}
\end{equation}
for $C=C(N,h_0,p,q)>0$, where we have again used using the boundedness assumption $\|u\|_{L^\infty(B_1)}\leq 1$ from \eqref{as1}. Therefore, using \eqref{lI16} and \eqref{lI16est}, for any $p\geq 2$, we obtain
\begin{equation}\label{lI16fest}
\begin{split}
I_{16}&\leq C\int_{B_R}|\nabla u_h|^q\,dx+C\int_{B_R}\frac{|\delta_h u|^\frac{\alpha q}{q-p+2}}{|h|^\frac{(1+\theta\alpha)q}{q-p+2}}\,dx,
\end{split}
\end{equation}
for $C=C(N,h_0,p,q)>0$. Similarly, we obtain
\begin{equation}\label{lI17fest}
\begin{split}
I_{17}&\leq C\int_{B_R}|\nabla u|^q\,dx+C\int_{B_R}\frac{|\delta_h u|^\frac{\alpha q}{q-p+2}}{|h|^\frac{(1+\theta\alpha)q}{q-p+2}}\,dx 
\end{split}
\end{equation}
for $C=C(N,h_0,p,q)>0$. Combining the estimates \eqref{lI16fest} and \eqref{lI17fest} in \eqref{lI15}, we have
\begin{equation}\label{lI15fest}
\begin{split}
I_{15}&\leq C\int_{B_R}|\nabla u_h|^q\,dx+C\int_{B_R}|\nabla u|^q\,dx+C\int_{B_R}\frac{|\delta_h u|^\frac{\alpha q}{q-p+2}}{|h|^\frac{(1+\theta\alpha)q}{q-p+2}}\,dx 
\end{split}
\end{equation}
for $C=C(N,h_0,p,q)>0$. Using the estimates \eqref{lI14} and \eqref{lI15fest} in \eqref{lIest} we have
\begin{equation}\label{lIfest}
\begin{split}
I&\geq cI_{13}-C\int_{B_R}|\nabla u_h|^q\,dx-C\int_{B_R}|\nabla u|^q\,dx-C\int_{B_R}\frac{|\delta_h u|^\frac{\alpha q}{q-p+2}}{|h|^\frac{(1+\theta\alpha)q}{q-p+2}}\,dx-C\\
&=c\int_{B_R}\Big|\nabla\Big(\frac{|u_h-u|^\frac{\alpha-1}{p}(u_h-u)\eta}{|h|^\frac{1+\theta\alpha}{p}}\Big)\Big|^p\,dx-C\int_{B_R}|\nabla u_h|^q\,dx\\
&-C\int_{B_R}|\nabla u|^q\,dx-C\int_{B_R}\frac{|\delta_h u|^\frac{\alpha q}{q-p+2}}{|h|^\frac{(1+\theta\alpha)q}{q-p+2}}\,dx-C
\end{split}
\end{equation}
for $c=c(p,\alpha)>0$ and $C=C(N,h_0,p,q,\alpha)>0$.
\\
\textbf{Step 3: Estimate of the nonlocal integral $J$.} First, we notice that
\begin{equation}\label{nJ}
\begin{split}
J&=J_1+J_2-J_3,
\end{split}
\end{equation}
where
\begin{align*}
J_1&=\int_{B_R}\int_{B_R}\frac{(J_p(u_h(x)-u_h(y))-(J_p(u(x)-u(y)))}{|h|^{1+\theta\alpha}}\\
&\times(J_{\alpha+1}(u_h(x)-u(x))\eta^p(x)-J_{\alpha+1}(u_h(y)-u(y))\eta^p(y))\,d\mu,
\end{align*}

\begin{align*}
J_2&=\int_{B_\frac{R+r}{2}}\int_{\mathbb{R}^N\setminus B_R}\frac{(J_p(u_h(x)-u_h(y))-(J_p(u(x)-u(y)))}{|h|^{1+\theta\alpha}}J_{\alpha+1}(u_h(x)-u(x))\eta^p(x)\,d\mu
\end{align*}
and
\begin{align*}
J_3&=-\int_{\mathbb{R}^N\setminus B_R}\int_{B_\frac{R+r}{2}}\frac{(J_p(u_h(x)-u_h(y))-(J_p(u(x)-u(y)))}{|h|^{1+\theta\alpha}}J_{\alpha+1}(u_h(y)-u(y))\eta^p(y)\,d\mu.
\end{align*}
\textbf{Estimate of $J_1$:} Proceeding exactly as in the proof of the estimate of $\mathcal{I}_1$ in \cite[Step 1, pages 813-817]{BLS}, we get
\begin{equation}\label{nJ1}
\begin{split}
J_1&=\int_{B_R}\int_{B_R}\frac{(J_p(u_h(x)-u_h(y))-(J_p(u(x)-u(y)))}{|h|^{1+\theta\alpha}}\\
&\times(J_{\alpha+1}(u_h(x)-u(x))\eta^p(x)-J_{\alpha+1}(u_h(y)-u(y))\eta^p(y))\,d\mu\\
&\geq c\left[\frac{|u_h-u|^\frac{\alpha-1}{p}(u_h-u)}{|h|^{1+\theta\alpha}}\eta\right]^p_{W^{s,p}(B_R)}-CJ_{11}-CJ_{12},
\end{split}
\end{equation}
for some constants $c=c(p,\alpha)>0$ and $C=C(p,\alpha)>0$, where
\begin{align*}
J_{11}&=\int_{B_R}\int_{B_R}\Big(|u_h(x)-u_h(y)|^\frac{p-2}{2}+|u(x)-u(y)|^\frac{p-2}{2}\Big)^2|\eta(x)^\frac{p}{2}-\eta(y)^\frac{p}{2}|^2\\
&\quad\times\frac{|u_h(x)-u(x)|^{\alpha+1}+|u_h(y)-u(y)|^{\alpha+1}}{|h|^{1+\theta\alpha}}\,d\mu
\end{align*}
and
\begin{align*}
J_{12}=\int_{B_R}\int_{B_R}\Big(\frac{|u_h(x)-u(x)|^{\alpha-1+p}}{|h|^{1+\theta\alpha}}+\frac{|u_h(y)-u(y)|^{\alpha-1+p}}{|h|^{1+\theta\alpha}}\Big)|\eta(x)-\eta(y)|^p\,d\mu.
\end{align*}
Proceeding along the lines of the proof of the estimates of $\mathcal{I}_{11}$\footnote{We remark that to estimate $J_{11}$ above in \eqref{nJ1}, we also used Lemma \ref{locnon} to estimate the fractional seminorm $[u]_{W^{\frac{s(p-2-\epsilon)}{p-2},q}(B_{R+h_0})}^q$ on page 818 in \cite{BLS}.} and $\mathcal{I}_{12}$ in \cite[Step 2, pages 817-819]{BLS}, we get
\begin{equation}\label{nJ11}
\begin{split}
|J_{11}|&\leq C\left(\int_{B_R}\frac{|\delta_h u|^\frac{\alpha q}{q-p+2}}{|h|^{(1+\theta\alpha)\frac{q}{q-p+2}}}\,dx+\int_{B_{R+4h_0}}|\nabla u|^q\,dx+1\right),
\end{split}
\end{equation}
and
\begin{equation}\label{nJ12}
\begin{split}
|J_{12}|&\leq C\left(\int_{B_R}\frac{|\delta_h u|^\frac{\alpha q}{q-p+2}}{|h|^{(1+\theta\alpha)\frac{q}{q-p+2}}}\,dx+1\right),
\end{split}
\end{equation}
where $C=C(N,h_0,p,s,q)>0$. Therefore, using the estimates \eqref{nJ11} and \eqref{nJ12} in \eqref{nJ1}, we have
\begin{equation}\label{nJ1est}
\begin{split}
J_1&\geq c\left[\frac{|u_h-u|^\frac{\alpha-1}{p}(u_h-u)}{|h|^{1+\theta\alpha}}\eta\right]^p_{W^{s,p}(B_R)}\\
&\quad-C\left(\int_{B_R}\frac{|\delta_h u|^\frac{\alpha q}{q-p+2}}{|h|^{(1+\theta\alpha)\frac{q}{q-p+2}}}\,dx+\int_{B_{R+4h_0}}|\nabla u|^q\,dx+1\right),
\end{split}
\end{equation}
for some constants $c=c(p,\alpha)>0$ and $C=C(N,h_0,p,s,q,\alpha)>0$.\\
\textbf{Estimates of $J_2$ and $J_3$:} Noting the assumptions in \eqref{as1} and then proceeding along the lines of the proof of the estimates of $\mathcal{I}_2$ and $\mathcal{I}_3$ in \cite[Step 3, pages 819-820]{BLS}, it follows that
\begin{equation}\label{nJ23}
\begin{split}
|J_2|+|J_3|&\leq C\left(1+\int_{B_R}\Big|\frac{\delta_h u}{|h|^\frac{1+\theta\alpha}{\alpha}}\Big|^\frac{\alpha q}{q-p+2}\,dx\right),
\end{split}
\end{equation}
where $C=C(N,h_0,s,p)>0$. Combining the estimates \eqref{nJ1est} and \eqref{nJ23} in \eqref{nJ}, we have
\begin{equation}\label{nJfest}
\begin{split}
J&\geq c\left[\frac{|u_h-u|^\frac{\alpha-1}{p}(u_h-u)}{|h|^{1+\theta\alpha}}\eta\right]^p_{W^{s,p}(B_R)}\\
&\quad-C\left(\int_{B_R}\Big|\frac{\delta_h u}{|h|^\frac{1+\theta\alpha}{\alpha}}\Big|^\frac{\alpha q}{q-p+2}\,dx+\int_{B_{R+4h_0}}|\nabla u|^q\,dx+1\right),
\end{split}
\end{equation}
for some constants $c=c(p,\alpha)>0$ and $C=C(N,h_0,p,s,q,\alpha)>0$.\\
\textbf{Step 4: Going back to the equation.} Inserting the estimates \eqref{lIfest} and \eqref{nJfest} in \eqref{tst3}, it follows that
\begin{equation}\label{IJest}
\begin{split}
\int_{B_R}\Big|\nabla\Big(\frac{|u_h-u|^\frac{\alpha-1}{p}(u_h-u)\eta}{|h|^\frac{1+\theta\alpha}{p}}\Big)\Big|^p\,dx&\leq C\Big(\int_{B_R}|\nabla u_h|^q\,dx+\int_{B_R}|\nabla u|^q\,dx+\int_{B_R}\Big|\frac{\delta_h u}{|h|^\frac{1+\theta\alpha}{\alpha}}\Big|^\frac{\alpha q}{q-p+2}\,dx+1\Big)\\
&+CA\Big(\int_{B_R}\Big|\frac{\delta_h u}{|h|^\frac{1+\theta\alpha}{\alpha}}\Big|^\frac{\alpha q}{q-p+2}\,dx+\int_{B_{R+4h_0}}|\nabla u|^q\,dx+1\Big).
\end{split}
\end{equation}
for some constant $C=C(N,h_0,p,s,q,\alpha)>0$. Next, we estimate the integral in the left hand side of the above inequality \eqref{IJest}. Indeed, we observe that following the lines of the proof of the estimate $(4.12)$ in \cite[page 821]{BLS} (one can run the same argument with $s=1$ there), we have the following estimate
\begin{equation}\label{gest}
\begin{split}
\Bigg\|\frac{\delta_{\xi}\delta_{h}u}{|\xi|^\frac{p}{\alpha-1+p}|h|^\frac{1+\theta\alpha}{\alpha-1+p}}\Bigg\|_{L^{\alpha-1+p}(B_r)}^{\alpha-1+p}&\leq C\Bigg\|\frac{\delta_{\xi}}{|\xi|}\Big(\frac{|\delta_h u|^\frac{\alpha-1}{p}(\delta_h u)\eta}{|h|^\frac{1+\theta\alpha}{p}}\Big)\Bigg\|_{L^p(\mathbb{R}^N)}^p\\
&\quad+C\Bigg\|\frac{\delta_{\xi}\eta}{|\xi|}\frac{\Big(|\delta_h u|^\frac{\alpha-1}{p}(\delta_h u)\Big)_{\xi}}{|h|^\frac{1+\theta\alpha}{p}}\Bigg\|_{L^p(\mathbb{R}^N)}^p,
\end{split}
\end{equation}
where $C=C(p,\alpha)>0$. Next, by Lemma \ref{lem:discdiffW1p} combined with the fact that $\eta$ is supported only in $B_R$
\begin{equation}\label{gest1}
\begin{split}
\sup_{|\xi|>0}\Bigg\|\frac{\delta_{\xi}}{|\xi|}\Big(\frac{|\delta_h u|^\frac{\alpha-1}{p}(\delta_h u)\eta}{|h|^\frac{1+\theta\alpha}{p}}\Big)\Bigg\|_{L^p(\mathbb{R}^n)}^p&\leq C \int_{B_R}\Big|\nabla\Big(\frac{|\delta_h u|^\frac{\alpha-1}{p}(\delta_h u)\eta}{|h|^\frac{1+\theta\alpha}{p}}\Big)\Big|^p\,dx,
\end{split}
\end{equation}
where $C=C(N,h_0,p)>0$. Noting the properties of $\eta$, the fact that $\|u\|_{L^\infty(B_1)}\leq 1$ from \eqref{as1} and using Young's inequality as in the proof of the estimate $(4.14)$ in \cite[pages 821-822]{BLS}, for any $0<|\xi|<h_0$, we get
\begin{equation}\label{gest2}
\begin{split}
\Bigg\|\frac{\delta_{\xi}\eta}{|\xi|}\frac{\Big(|\delta_h u|^\frac{\alpha-1}{p}(\delta_h u)\Big)_{\xi}}{|h|^\frac{1+\theta\alpha}{p}}\Bigg\|_{L^p(\mathbb{R}^N)}^p&\leq C\Big(\int_{B_R}\Big|\frac{\delta_h u}{|h|^\frac{1+\theta\alpha}{\alpha}}\Big|^\frac{q\alpha}{q-p+2}\,dx+1\Big),
\end{split}
\end{equation}
where $C=C(N,h_0,p)>0$. Combining \eqref{gest1} and \eqref{gest2} in \eqref{gest}, for every $0<|\xi|<h_0$, we have
\begin{equation}\label{gest3}
\begin{split}
\Bigg\|\frac{\delta_{\xi}\delta_{h}u}{|\xi|^\frac{p}{\alpha-1+p}|h|^\frac{1+\theta\alpha}{\alpha-1+p}}\Bigg\|_{L^{\alpha-1+p}(B_r)}^{\alpha-1+p}&\leq  C \int_{B_R}\Big|\nabla\Big(\frac{|\delta_h u|^\frac{\alpha-1}{p}(\delta_h u)\eta}{|h|^\frac{1+\theta\alpha}{p}}\Big)\Big|^p\,dx\\
&\quad+C\Big(\int_{B_R}\Big|\frac{\delta_h u}{|h|^\frac{1+\theta\alpha}{\alpha}}\Big|^\frac{q\alpha}{q-p+2}\,dx+1\Big),
\end{split}
\end{equation}
where $C=C(N,h_0,p,\alpha)>0$. Choosing $\xi=h$ and taking supremum over $h$ for $0<|h|<h_0$ and then using \eqref{gest3} in \eqref{IJest}, it follows that
\begin{equation}\label{gfest}
\begin{split}
&\sup_{0<|h|<h_0}\int_{B_r}\Big|\frac{\delta_h^{2}u}{|h|^\frac{1+p+\theta\alpha}{\alpha-1+p}}\Big|^{\alpha-1+p}\,dx\\
&\quad\leq C\Big(\sup_{0<|h|<h_0}\int_{B_{R}}|\nabla u_h|^q\,dx+\int_{B_{R}}|\nabla u|^q\,dx+\sup_{0<|h|<h_0}\int_{B_R}\Big|\frac{\delta_h u}{|h|^\frac{1+\theta\alpha}{\alpha}}\Big|^\frac{\alpha q}{q-p+2}\,dx+1\Big),\\
&\qquad+CA\Big(\sup_{0<|h|<h_0}\int_{B_R}\Big|\frac{\delta_h u}{|h|^\frac{1+\theta\alpha}{\alpha}}\Big|^\frac{\alpha q}{q-p+2}\,dx+\int_{B_{R+4h_0}}|\nabla u|^q\,dx+1\Big)\\
&\quad\leq C\Big(\int_{B_{R+h_0}}|\nabla u|^q\,dx+\sup_{0<|h|<h_0}\int_{B_R}\Big|\frac{\delta_h u}{|h|^\frac{1+\theta\alpha}{\alpha}}\Big|^\frac{\alpha q}{q-p+2}\,dx+1\Big)\\
&\qquad+CA\Big(\sup_{0<|h|<h_0}\int_{B_R}\Big|\frac{\delta_h u}{|h|^\frac{1+\theta\alpha}{\alpha}}\Big|^\frac{\alpha q}{q-p+2}\,dx+\int_{B_{R+4h_0}}|\nabla u|^q\,dx+1\Big),
\end{split}
\end{equation}
where $C=C(N,h_0,p,q,s,\alpha)>0$.\\
\textbf{Step 5: Conclusion.} Now we set,
\begin{align*}
\alpha=q-p+2,\quad \theta=\frac{q-p+1}{q-p+2}.
\end{align*}
Therefore, we obtain
\begin{align*}
\frac{1+p+\theta\alpha}{\alpha-1+p}=\frac{1}{q+1}+1,\quad \alpha-1+p=q+1,\quad \frac{q\alpha}{q-p+2}=q,\quad\frac{1+\theta\alpha}{\alpha}=1.
\end{align*}
Plugging these values in \eqref{gfest}, we finally deduce that
\begin{equation}\label{gfest1}
\begin{split}
\sup_{0<|h|<h_0}\Bigg\|\frac{\delta_h^{2}u}{|h|^{\frac{1}{q+1}+1}}\Bigg\|^{q+1}_{L^{q+1}(B_r)}&\leq C\Big(\int_{B_{R+h_0}}|\nabla u|^q\,dx+\sup_{0<|h|<h_0}\Bigg\|\frac{\delta_h u}{|h|}\Bigg\|^{q}_{L^q(B_{R})}+1\Big)\\
&\quad+CA\Big(\sup_{0<|h|<h_0}\Bigg\|\frac{\delta_h u}{|h|}\Bigg\|^{q}_{L^q(B_{R})}+\int_{B_{R+4h_0}}|\nabla u|^q\,dx+1\Big),
\end{split}
\end{equation}
where $C=C(N,h_0,p,q,s)>0$. In particular, recalling that $r=R-4h_0$ and using Theorem 3 on page 277 in \cite{Evans} to estimate the difference quotients, \eqref{gfest1} gives
\begin{equation}\label{gfest2}
\begin{split}
\sup_{0<|h|<h_0}\Bigg\|\frac{\delta_h^{2}u}{|h|^{1+\frac{1}{q+1}}}\Bigg\|^{q+1}_{L^{q+1}(B_{R-4h_0})}&\leq C(1+A)\Bigg(\int_{B_{R+4h_0}}|\nabla u|^q\,dx+1\Bigg),
\end{split}
\end{equation}
where $C=C(N,h_0,p,q,s)>0$.
\end{proof}
\begin{Lemma}\label{locsemi}(Estimate of the local seminorm) Let $2\leq p<\infty,\,0<s<1$ and $0\leq A\leq 1$. Suppose $u\in W^{1,p}_{\mathrm{loc}}(B_2(x_0))\cap L^{p-1}_{sp}(\mathbb{R}^N)$ is a weak solution of 
$$
-\Delta_p u+A(-\Delta_p)^s u =0 \text{ in $B_2(x_0)$}
$$
satisfying
$$
\|u\|_{L^\infty(B_1(x_0))}\leq 1,\qquad \int_{\mathbb{R}^N\setminus B_1(x_0)}\frac{|u|^{p-1}}{|x|^{N+s\,p}}\, dx\leq 1. 
$$
Then 
$$
\int_{B_\frac{7}{8}(x_0)}|\nabla u|^p dx\leq C(N,p,s).
$$
\end{Lemma}
\begin{proof} Without loss of generality, we assume $x_0=0$. We only provide the proof for $w=u^+$, the proof of $u^-$ is similar. We apply \cite[Lemma 3.1]{GK} with $r=1$, $x_0=0$ and with $\psi\in C_0^\infty(B_\frac89)$ such that $\psi =1$ on $B_\frac78$, $0\leq \psi \leq 1$ and $|\nabla \psi| \leq C$ for some $C=C(N)>0$. By using the properties of $\psi$ and $A\in[0,1]$ this yields
\[
\begin{split}
\int_{B_\frac78}|\nabla w|^p dx &\leq  C(N,p)\left(\int_{B_1}|w|^p dx+\int_{B_1}\int_{B_1}\left(|w(x)|^p+|w(y)|^p\right)|x-y|^{1-sp-N} dx dy\right)
\\&+C(N,p)\int_{\R^N\setminus B_1}\frac{|w(y)|^{p-1}}{|y|^{N+sp}} dy\int_{B_1} |w (x)| dx\\
&\leq C(N,p)(1+C(p,s)+1).
\end{split}
\]
Hence the result follows.
\end{proof}
We are now ready to prove Theorem \ref{teo:1}.
\begin{proof}[~Proof of Theorem \ref{teo:1}]
We first observe that $u\in L^\infty_{\rm loc}(\Omega)$, by \cite[Theorem 4.2]{GK}. We assume for simplicity that $x_0=0$, then we set 
\[
\mathcal{M}_R=\|u\|_{L^\infty(B_{R})}+\mathrm{Tail}_{p-1,s\,p,s\,p}(u;0,R)>0.
\]
We point out that it is sufficient to prove that the rescaled function
\begin{equation}
\label{eq:uR}
u_R(x):=\frac{1}{\mathcal{M}_R}\,u(R\,x),\qquad \mbox{ for }x\in B_2,
\end{equation}
satisfy the estimate
\[
[u_R]_{C^{\delta}(B_{1/2})}\leq C.
\]
By scaling back, we would get the desired estimate. Observe that by definition, the function $u_R$ 
is a local weak solution of $-\Delta_p u+A\,R^{p-ps}(-\Delta_p)^s u=0$ in $B_2$ and satisfies
\begin{equation}
\label{assumption}
\|u_R\|_{L^\infty(B_1)}\leq 1,\qquad \int_{\mathbb{R}^N\setminus B_1}\frac{|u_R(y)|^{p-1}}{|y|^{N+s\,p}}\,  dy\leq 1,\qquad [u_R]_{W^{1,p}(B_\frac78)}\leq C(N,p,s).
\end{equation}
The last estimate follows from Lemma \ref{locsemi}.
In what follows, we will omit the subscript $R$ and simply write $u$ in place of $u_R$, in order not to overburden the presentation.
\vskip.2cm\noindent
We fix $0<\delta<1$ and choose $i_\infty\in\mathbb{N}\setminus\{0\}$ such that
\[
1-\delta> \frac{N}{p+i_\infty}.
\]
Then we define the sequence of exponents
\[
q_i=p+i,\qquad i=0,\dots,i_\infty.
\]
We define also	
$$
h_0=\frac{1}{112\,i_\infty},\qquad R_i=\frac{7}{8}-4\,h_0-14\,h_0i,\qquad \mbox{ for } i=0,\dots,i_\infty.
$$
We note that 
\[
R_0+4\,h_0=\frac{7}{8}\qquad\mbox{ and }\qquad R_{i_\infty}+4\,h_0=\frac{3}{4}.
\] 
By applying Proposition \ref{prop1} with\footnote{We observe that by construction we have
\[
4\,h_0<R_i\le 1-5\,h_0,\qquad \mbox{ for } i=0,\dots,i_\infty.
\]
Thus these choices are admissible in Proposition \ref{prop1}.} 
\[
R=R_i\qquad \mbox{ and }\qquad q=q_i=p+i,\qquad \mbox{ for } i=0,\ldots,i_\infty,
\] 
and by \eqref{assumption} along with $A\in [0,1]$ and $R\in (0,1)$, we obtain 
\begin{equation}\label{it1}
\begin{split}
\sup\limits_{0<|h|< h_0}\left\|\dfrac{\delta^2_h u}{|h|^{1+\frac{1}{q_1}}}\right\|^{q_1}_{L^{q_1}(B_{R_0-4h_0})}&\leq C\,\left([u]_{W^{1,p}(B_{\frac{7}{8}})}^{p}+1\right)\leq C(N,p,s,\delta).
\end{split}
\end{equation}
Noting that $R_i-10 h_0=R_{i+1}+4h_0$ for every $i=0,1,\ldots,i_{\infty}-1$ and using Lemma \ref{2.4} in \eqref{it1}, we get
\begin{equation}\label{it1app}
\begin{split}
[u]_{W^{1,q_1}(B_{R_1+4h_0})}^{q_1}&\leq C(N,p,s,\delta).
\end{split}
\end{equation}
Again, by Proposition \ref{prop1} and applying \eqref{it1app}, we obtain
\begin{equation}\label{ite2}
\begin{split}
\sup\limits_{0<|h|< h_0}\left\|\dfrac{\delta^2_h u}{|h|^{1+\frac{1}{q_2}}}\right\|^{q_2}_{L^{q_2}(B_{R_{1}-4h_0})}&\leq C\,\left([u]_{W^{1,q_1}(B_{R_1+4h_0})}^{q_1}+1\right)\leq C(N,p,s,\delta).
\end{split}
\end{equation}
Further, using Lemma \ref{2.4} in \eqref{ite2}, we get
\begin{equation}\label{it2app}
\begin{split}
[u]_{W^{1,q_2}(B_{R_2+4h_0})}^{q_2}&\leq C(N,p,s,\delta).
\end{split}
\end{equation}
Repeating this procedure, we  obtain the iteration scheme
\begin{equation}\label{ite3}
\begin{split}
[u]^{q_{i+1}}_{W^{1,q_{i+1}}(B_{R_{i+1} +4h_0})}&\leq C(N,p,s,\delta),
\end{split}
\end{equation}
for all $i=0,1,\ldots,i_\infty-1$. Choosing $i=i_{\infty}-1$ in \eqref{ite3} and using the facts that $\|u\|_{L^{\infty}(B_1)}\leq 1,\,\,[u]_{W^{1,p}(B_1)}\leq 1$, we obtain 
$$
\|u\|_{{W^{1,q_{i_\infty}}(B_{R_{i_\infty}+4h_0})}}\leq C
$$
for $C=C(N,p,s,\delta)>0$. Since $q_{i_\infty}>N$ and $R_{i_\infty}+4h_0=\frac{3}{4}$, by Morrey's embedding theorem, we get $u\in C^{\delta}_{\mathrm{loc}}(B_{\frac{3}{4}})$ and
$$
[u]_{C^{\delta}(B_{\frac{1}{2}})}\leq C\|u\|_{{W^{1,q_{i_\infty}}(B_{R_{i_\infty}+4h_0})}}\leq C
$$
for $C=C(N,p,s,\delta)>0$. Since $\delta\in(0,1)$ is arbitrary, the result follows.
\end{proof}
Since the result above implies that the nonlocal term is bounded when $sp<(p-1)$, we can finally give the proof of Corollary \ref{corc1}.

\begin{proof}[Proof of Corollary \ref{corc1}.]
Upon rescaling as in the proof of Theorem \ref{teo:1}, it is sufficient to prove that $\|u_R\|_{C^{1,\alpha}(B_\frac18)}\leq C$ with $u_R$ as defined in \eqref{eq:uR} satisfying 
\begin{equation}
\label{assumption2}
\|u_R\|_{L^\infty(B_1)}\leq 1,\qquad \int_{\mathbb{R}^N\setminus B_1}\frac{|u_R(y)|^{p-1}}{|y|^{N+s\,p}}\,  dy\leq 1.
\end{equation}
Theorem \ref{teo:1} implies that there is $\delta>sp/(p-1)$ such that
\[
[u_R]_{C^{\delta}(B_{1/2})}\leq C(N,s,p,\delta).
\]
Now take any $x_0\in B_{1/4}$. Then 
$$
\int_{B_\frac14(x_0)}\frac{|u_R(x_0)-u_R(y)|^{p-1}}{|x_0-y|^{N+sp}} dy\leq C\int_{B_\frac14(x_0)}|x_0-y|^{N+sp-\delta(p-1)} dy=C(N,s,p,\delta),
$$
by the choice of $\delta$. Moreover, 
$$
\int_{\R^N\setminus B_\frac14(x_0)}\frac{|u_R(x_0)-u_R(y)|^{p-1}}{|x_0-y|^{N+sp}} dy\leq C(N,s,p),
$$
by \eqref{assumption2}. Hence, $\|(-\Delta_p)^s u_R\|_{L^\infty(B_{1/4})}\leq C(N,s,p,\delta)$ and therefore also $\|\Delta_p u_R\|_{L^\infty(B_{1/4})}\leq C(N,s,p,\delta)$ which together with \eqref{assumption2} and the well known $C^{1,\alpha}$-estimates for the $p$-Laplacian (see for instance the corollary on page 830 in \cite{DiB}) imply
$$
\|u_R\|_{C^{1,\alpha}(B_\frac18)}\leq C(N,s,p,\delta).
$$
\end{proof}
\normalcolor
\section{Regularity for the inhomogeneous equation} \label{sec:inhomo}
In this section, we prove the boundedness and the regularity for the inhomogeneous equation.

\subsection{Boundedness}
We now address the boundedness, by comparing with the homogeneous equation. The first one is a consequence of Sobolev's inequality.
\begin{Lemma}\label{lem:w1p} Let $2\leq p<\infty,\,0<s<1$ and $A\geq 0$. Suppose $f\in L^q(\Omega)$ for $q>N/p$ if $p\leq N$ and $q\geq 1$ otherwise. Assume that $u\in W^{1,p}_{\mathrm{loc}}(\Omega)\cap L^{p-1}_{sp}(\mathbb{R}^N)$ is a weak subsolution of
$$
-\Delta_p u +A(-\Delta_p)^s u =f \text{ in $\Omega$}
$$
such that $B_r(x_0)\Subset \Omega$ and that $v\in W^{1,p}_{u}(B_r(x_0))$ solves
$$
\begin{cases}
-\Delta_p v +A(-\Delta_p)^s v =0 &\text{ in $B_r(x_0)$},\\
v = u & \text{ in $\R^N\setminus B_r(x_0)$}.
\end{cases}
$$
Then
\begin{equation}\label{bd1}
\|(u-v)^+\|_{L^{q'}(B_r(x_0))}\leq 2^\frac{p-2}{p-1}\left(\|f\|_{L^q(B_r(x_0))}S_{N,p}\right)^\frac{1}{(p-1)}|B_r(x_0)|^{\frac{p}{p-1}(\frac{1}{q'}-\frac{1}{p}+\frac{1}{N})},
\end{equation}

\begin{equation}\label{bd2}
\|\nabla (u-v)^+\|_{L^p(B_r(x_0))}\leq 2^\frac{p-2}{p-1}\|f\|_{L^q(B_r(x_0))}^\frac{1}{p-1}S_{N,p}^\frac{1}{p(p-1)}|B_r(x_0)|^{\frac{1}{p-1}(\frac{1}{q'}-\frac{1}{p}+\frac{1}{N})},
\end{equation}
and
\begin{equation}\label{bd3}
\|(u-v)^+\|_{L^p(B_r(x_0))}\leq 2^\frac{p-2}{p-1}C(N,p)\|f\|_{L^q(B_r(x_0))}^\frac{1}{p-1}S_{N,p}^\frac{1}{p(p-1)}|B_r(x_0)|^{\frac{1}{p-1}(\frac{1}{q'}-\frac{1}{p}+\frac{1}{N}){+\frac{1}{N}}}.
\end{equation}
Here, $S_{N,p}$ is the constant in the Sobolev embedding in $W^{1,p}$.
\end{Lemma}
\begin{proof}
By Sobolev's inequality and H\"older's inequality we have
\begin{equation}\label{bd4}
\|(u-v)^+\|_{L^{q'}(B_r(x_0))}\leq \left(S_{N,p}\right)^\frac{1}{p}\|\nabla (u-v)\|_{L^p(B_r(x_0))}|B_r(x_0)|^{\frac{1}{q'}-\frac{1}{p}+\frac{1}{N}}.
\end{equation}
We test the difference of the equations for $u$ and $v$ with $(u-v)^+$ and observe that by Lemma \ref{lem:pineq} and some manipulations, we have
$$
\int_{\mathbb{R}^N}\int_{\mathbb{R}^N}(J_p(u(x)-u(y))-J_p(v(x)-v(y))){((u-v)^+(x)-(u-v)^+(y))}\,d\mu\geq 0.
$$
Therefore, we may throw away the nonlocal term. We obtain from H\"older's inequality and Lemma \ref{lem:pineq}
\begin{equation}\label{bd5}
2^{2-p}\int_{B_r(x_0)}|\nabla (u-v)^+|^p dx\leq \|f\|_{L^q(B_r(x_0))}\|(u-v)^+\|_{L^{q'}(B_r(x_0))}.
\end{equation}
The two inequalities \eqref{bd4} and \eqref{bd5} together imply \eqref{bd1} and \eqref{bd2}. Finally, using Poincar\'e's inequality and \eqref{bd2}, we obtain \eqref{bd3}.
\end{proof}
We now perform a Moser iteration to obtain the boundedness.
\begin{prop}[$L^\infty$-estimate]
\label{prop:Linfty}
Let $2\leq p<\infty,\,0<s<1$ and $A\geq 0$. Suppose $f\in L^q(\Omega)$ for $q>N/p$ if $p\leq N$ and $q\geq 1$ otherwise. Assume that $u\in W^{1,p}_{\mathrm{loc}}(\Omega)\cap L^{p-1}_{sp}(\mathbb{R}^N)$ is a weak subsolution of
\begin{equation}\label{lsub}
-\Delta_p u +A(-\Delta_p)^s u =f \text{ in $\Omega$}
\end{equation}
such that $B_r(x_0)\Subset \Omega$ and that $v$ solves
\begin{equation}\label{nsub}
\begin{cases}
-\Delta_p v +A(-\Delta_p)^s v =0 &\text{ in $B_r(x_0)$},\\
v = u & \text{ in $\R^N\setminus B_r(x_0)$}.
\end{cases}
\end{equation}
Then $(u-v)^+\in L^\infty(B_r(x_0))$, with the following estimate
\[
\|(u-v)^+\|_{L^\infty(B_r(x_0))}\le C(N,p,q) \left(|B_r(x_0)|^{\frac{p}{N}-\frac{1}{q}}\|f\|_{L^q(B_r(x_0))}\right)^\frac{1}{p-1}.
\]
\end{prop}
\begin{proof} For simplicity, we assume that $x_0=0$. We follow closely the proof of Theorem 3.1 in \cite{BP16}. We first note that if $p>N$ then by Morrey's inequality
$$
\|(u-v)^+\|_{L^\infty(B_r)}\le C(N,p)|B_r|^{\frac{1}{N}-\frac{1}{p}}\|\nabla (u-v)^+\|_{L^p(B_r)}.
$$
This together with Lemma \ref{lem:w1p} implies
$$
\|(u-v)^+\|_{L^\infty(B_r)}\le C(N,p)\|f\|^\frac{1}{p-1}_{L^q(B_r)}|B_r|^{\frac{1}{p-1}(\frac{1}{q'}-\frac{1}{p}+\frac{1}{N})}|B_r|^{\frac{1}{N}-\frac{1}{p}}
$$
which is the desired result.

We now prove the result for the positive part of $u-v$ in the case $p< N$ and then comment on how the proof would be changed if $p=N$. Let $w=(u-v)^+$,  $\delta>0$ and $\beta> 1$. We observe that $u+\delta$ is again a weak subsolution of \eqref{lsub}. Insert the test function\footnote{This function is not really admissible but it can be made rigorous by instead taking $\min(w,M)$ for some $M>0$ and then letting $M\to \infty$.}
\[
\varphi=(w+\delta)^{\beta}-\delta^\beta
\]
in the difference of the equations for $u+\delta$ and $v$. The part coming from the nonlocal part will be non-negative. Indeed, this part is given by
\[
\int_{\mathbb{R}^N}\int_{\mathbb{R}^N}\left(J_p(u(x)+\delta-(u(y)+\delta))-J_p(v(x)-v(y)))\right){((w+\delta)^\beta(x)-(w+\delta)^\beta(y))}\,d\mu\geq 0.
\]
by Lemma \ref{lem:powertest} and some manipulations. For the local term we will apply
 Lemma \ref{lem:pineq}. This gives
\[
2^{2-p}\frac{\beta\,p^p}{(\beta+p-1)^p}\,\int_{B_r}\left|\nabla (w+\delta)^\frac{\beta+p-1}{p}\right|^p\,dx\,\leq\int_{B_r} f \varphi\, dx \le \|f\|_{L^q(B_r)} \|(w+\delta)^{\beta}\|_{L^{q'}(B_r)}.
\]
By observing that for every $\beta\ge 1$ we have
\[
\left(\frac{\beta+p-1}{p}\right)^p\,\frac{1}{\beta}\le \left(\beta\right)^{p-1},
\]
we can rewrite the previous estimate as
\[
\int_{B_r}\left|\nabla (w+\delta)^\frac{\beta+p-1}{p}\right|^p\,dx\le 2^{p-2}\left(\beta\right)^{p-1}\,\|f\|_{L^q(B_r)} \|(w+\delta)^{\beta}\|_{L^{q'}(B_r)}.
\]
With $\vartheta=(\beta+p-1)/p$, the previous inequality is equivalent to
\begin{equation}
\label{moser}
\int_{B_r}\left|\nabla (w+\delta)^\vartheta\right|^p\,dx\le 2^{p-2}\beta^{p-1}\,\|f\|_{L^q(B_r)} \|(w+\delta)^{\beta}\|_{L^{q'}{(B_r)}}.
\end{equation}
We now proceed using the Sobolev inequality:
\[
\left(\int_\Omega |\varphi|^{p^*}\,dx\right)^\frac{p}{p^*}\le S_{N,p} \int_\Omega |\nabla \varphi|^p\,dx,\qquad \mbox{ for every } \varphi\in W^{1,p}_0(\Omega),
\]
where $p^*=(N\,p)/(N-p)$.
By using this inequality in the left-hand side of \eqref{moser}, we get
\[
\,\left(\int_{B_r} \big |(w+\delta)^{\vartheta}-\delta^\vartheta\big|^{p^*}\,dx\right)^\frac{p}{p^*}\le S_{N,p}2^{p-2}\beta^{p-1}\|f\|_{L^q(B_r)} \|(w+\delta)^{\beta}\|_{L^{q'}{(B_r)}}
\]
and thus
$$
\|{(w+\delta)}^{\vartheta}-\delta^\vartheta\|_{L^{p^*}(B_r)}\leq \left(2^{p-2}\beta^{p-1}\|f\|_{L^q(B_r)}S_{N,p}\right)^\frac{1}{p} \|(w+\delta)^{\beta}\|^\frac{1}{p}_{L^{q'}(B_r)}.
$$
By the triangle inequality
\[
\begin{split}
&\|(w+\delta)^\beta-\delta^\vartheta\|_{L^{p^*}(B_r)}\geq\|\delta^{\frac{p-1}{p}}((w+\delta)^\frac{\beta}{p}-\delta^\frac{\beta}{p})\|_{L^{p^*}(B_r)}
\\
&\geq \delta^{\frac{p-1}{p}}\|((w+\delta)^\frac{\beta}{p}\|_{L^{p^*}(B_r)}-\delta^\vartheta|B_r|^\frac{1}{p^*}.
\end{split}
\]
Therefore, since $\vartheta=(\beta+p-1)/p$, we obtain
$$
\|(w+\delta)^\frac{\beta}{p}\|_{L^{p^*}(B_r)}\leq \frac{1}{\delta^\frac{p-1}{p}}\left(2^{p-2}\beta^{p-1}\|f\|_{L^q(B_r)}S_{N,p}\right)^\frac{1}{p} \|(w+\delta)^{\beta}\|^\frac{1}{p}_{L^{q'}(B_r)}+\delta^{\frac{\beta}{p}}|B_r|^\frac{1}{p^*}.
$$
Using that $\beta\geq 1$ we also have
$$
\delta^\beta = \|\delta^\beta\|_{L^{q'}(B_r)}|B_r|^{-\frac{1}{q'}}\leq \beta^{(p-1)}\|(w+\delta)^\beta\|_{L^{q'}(B_r)}|B_r|^{-\frac{1}{q'}}.
$$
Therefore, 
$$
\delta^{\frac{\beta}{p}}|B_r|^\frac{1}{p^*}\leq \beta^\frac{p-1}{p}\|(w+\delta)^\beta\|^\frac{1}{p}_{L^{q'}(B_r)}|B_r|^{-\frac{1}{pq'}+\frac{1}{p^*}}
$$
so that
$$
\|(w+\delta)^\frac{\beta}{p}\|_{L^{p^*}(B_r)}\leq   2^{(p-2)/p}\beta^{(p-1)/p}\|(w+\delta)^{\beta}\|^\frac{1}{p}_{L^{q'}(B_r)}\left(\frac{1}{\delta^\frac{p-1}{p}}\left(\|f\|_{L^q(B_r)}S_{N,p}\right)^\frac{1}{p}+{2^\frac{2-p}{p}}|B_r|^{-\frac{1}{pq'}+\frac{1}{p^*}}\right).
$$
Now we make the choice
$$
\delta = \left({2^{p-2}}\|f\|_{L^q(B_r)}S_{N,p}\right)^\frac{1}{p-1}|B_r|^{-\frac{p}{p-1}(-\frac{1}{pq'}+\frac{1}{p^*})}.
$$
Then we obtain the estimate
$$
\|(w+\delta)^\frac{\beta}{p}\|_{L^{p^*}(B_r)}\leq |B_r|^{\frac{1}{p^*}-\frac{1}{pq'}}\beta^{(p-1)/p}\|(w+\delta)^{\beta}\|^\frac{1}{p}_{L^{q'}(B_r)}
$$
or with the notation $\gamma = \beta q'$ and $\chi = p^*/(pq')>1$
\[
\begin{split}
\|w+\delta\|_{L^{\chi \gamma}(B_r)}&\leq   \left(|B_r|^{(\frac{p}{p^*}-\frac{1}{q'})}\right)^\frac{q'}{\gamma}\left(\frac{\gamma}{q'}\right)^{\frac{q'(p-1)}{\gamma}}\|(w+\delta)\|_{L^{\gamma}(B_r)} \\
&= \left(|B_r|^{(1-\frac{p}{N}-\frac{1}{q'})}\right)^\frac{q'}{\gamma}\left(\frac{\gamma}{q'}\right)^{\frac{q'(p-1)}{\gamma}}\|(w+\delta)\|_{L^{\gamma}(B_r)}.
\end{split}
\]
Now it is just a matter of following the exact same steps as in the proof of Theorem 3.1 in Brasco-Parini \cite{BP16} with $s=1$. Here we make the choices
$$
\gamma_0=q', \quad \gamma_n=\chi^nq'.
$$
Then
$$
{\sum_{n=0}^{\infty}} \frac{q'}{\gamma_n} = {\sum_{n=0}^{\infty}} \chi^{n} =\frac{\chi}{\chi-1}= \frac{N}{N-q'+pq'}
$$
and
$$
{\prod_{n=0}^{\infty}} \left(\frac{\gamma_n}{q'}\right)^\frac{q'}{\gamma_n} = \chi^\frac{\chi}{(\chi-1)^2}.
$$
The final estimate becomes
$$
\|w+\delta\|_{L^\infty(B_r)}\leq (C)^\frac{\chi}{\chi-1}(\chi^{p-1})^\frac{\chi}{(\chi-1)^2}\left(|B_r|^{(1-\frac{p}{N}-\frac{1}{q'})}\right)^\frac{\chi}{\chi-1}\|w+\delta\|_{L^{q'}(B_r)},
$$
for some constant $C=C(p)>0$.
Therefore
$$
\|w\|_{L^\infty(B_r)}\leq (C)^\frac{\chi}{\chi-1}(\chi^{p-1})^\frac{\chi}{(\chi-1)^2}\left(|B_r|^{(1-\frac{p}{N}-\frac{1}{q'})}\right)^\frac{\chi}{\chi-1}\left(\|w\|_{L^{q'}(B_r)}+\delta|B_r|^\frac{1}{q'}\right).
$$
By the choice of $\delta$ this becomes
$$
\|w\|_{L^\infty(B_r)}\leq (C)^\frac{\chi}{\chi-1}(\chi^{p-1})^\frac{\chi}{(\chi-1)^2}\left(|B_r|^{-\frac{1}{q'}}\|w\|_{L^{q'}(B_r)}+
\left({2^{p-2}}|B_r|^{\frac{p}{N}-\frac{1}{q}}\|f\|_{L^q(B_r)}{S_{N,p}}\right)^\frac{1}{p-1}\right).
$$
By the estimate \eqref{bd1} in Lemma \ref{lem:w1p} we obtain 
\[
\begin{split}
&\|w\|_{L^\infty(B_r)}\\
&\leq (C)^\frac{\chi}{\chi-1}(\chi^{p-1})^\frac{\chi}{(\chi-1)^2}\left(2^\frac{p-2}{p-1}\|f\|_{L^q(B_r)}^{\frac{1}{p-2}}S_{N,p}^\frac{1}{(p-1)}|B_r|^{\frac{p}{p-1}(\frac{1}{q'}-\frac{1}{p}+\frac{1}{N})-\frac{1}{q'}}+
\left(|B_r|^{\frac{p}{N}-\frac{1}{q}}\|f\|_{L^q(B_r)}{S_{N,p}}\right)^\frac{1}{p-1}\right)\\
&\leq C(N,p,q) \left(|B_r|^{\frac{p}{N}-\frac{1}{q}}\|f\|_{L^q(B_r)}{S_{N,p}}\right)^\frac{1}{p-1}.
\end{split}
\]
{\bf Comment on the case $p=N$.} For the case $p=N$, we simply replace the Sobolev embedding with the embedding inequality of $W_0^{1,N}(B_r)$ into $L^{q}$ for $q$ large.
\end{proof}

\begin{proof}[~Proof of Theorem \ref{teo:loc_bound}] We may assume $x_0=0$. Upon using the rescaling $x\mapsto Rx$ it is also enough to prove the estimate
$$
\|u^+\|_{L^\infty(B_{\sigma })}\leq C(N,p,s,\sigma)\,\left[\left(\fint_{B_{1}} |u^+|^p\, dx\right)^\frac{1}{p}+\mathrm{Tail}_{p-1,s\,p,s\,p}(u^+;0,1 )+\|f_R\|_{L^q(B_{1})}^\frac{1}{p-1}\right],
$$
for a solution of 
$$
-\Delta_p u+AR^{p-sp}(-\Delta_p)^s\,u=f_R(x):=R^p f(Rx)
$$
in $B_1$. Take $\rho=(1-\sigma)/2+\sigma$ and let $v$ be the solution of
\[
\left\{\begin{array}{rcll}
-\Delta_p v+AR^{p-sp}(-\Delta_p)^s\,v&=&0,&\mbox{ in }B_{\rho},\\
v&=&u,& \mbox{ in }\mathbb{R}^N\setminus B_{\rho}.
\end{array}
\right.
\]
By Proposition \ref{prop:Linfty} 
\[
\|(u-v)^+\|_{L^\infty(B_{\rho})}\le C(N,p,q) \left(|B_{\rho}|^{\frac{p}{N}-\frac{1}{q}}\|f_R\|_{L^q(B_{\rho}))}\right)^\frac{1}{p-1}=C(N,p,q,\sigma)\|f_R\|_{L^q(B_{\rho})}^\frac{1}{p-1}.
\]
Moreover, by \cite[Theorem 4.2]{GK} (note that $R^{p-sp}< 1$, since $R<1$)
$$
\|v^+\|_{L^\infty(B_{\sigma })}\le C(N,p,s)\,\left[\left(\fint_{B_{2\sigma}} |v^+|^p\, dx\right)^\frac{1}{p}+\mathrm{Tail}_{p-1,s\,p,s\,p}(v^+;0,\sigma )\right],
$$
Therefore, 
\[
\begin{split}
\|u^+\|_{L^\infty(B_{\sigma })}&\le \|v^+\|_{L^\infty(B_{\sigma })}+\|(u-v)^+\|_{L^\infty(B_{\sigma })}
\\
&\leq C(N,p,s)\,\left[\left(\fint_{B_{2\sigma}} |v^+|^p\, dx\right)^\frac{1}{p}+\mathrm{Tail}_{p-1,s\,p,s\,p}(v^+;0,\sigma)\right]+C(N,p,q,\sigma)\|f_R\|_{L^q(B_{\rho})}^\frac{1}{p-1}\\
&\leq C(N,p,q,s,\sigma)\,\left[\left(\fint_{B_{\rho}} |v^+|^p\, dx\right)^\frac{1}{p}+\mathrm{Tail}_{p-1,s\,p,s\,p}(v^+;0,\rho)+\|f_R\|_{L^q(B_{\rho})}^\frac{1}{p-1}\right]\\
&\leq C(N,p,q,s,\sigma)\,\left[\left(\fint_{B_{1}} |u^+|^p\, dx\right)^\frac{1}{p}+\mathrm{Tail}_{p-1,s\,p,s\,p}(u^+;0,1 )+\|f_R\|_{L^q(B_{1})}^\frac{1}{p-1}\right],
\end{split}
\]
where we used Lemma \ref{lem:w1p} to estimate the $L^p$-norm of $v^+$ in terms of the $L^p$-norm of $u^+$ and the fact that $u=v$ outside $B_{\rho}$ to estimate the tail term. This is the desired result.
\end{proof}

\subsection{Higher H\"older regularity}
Here we turn our attention to the regularity of the inhomogenous equation. We first establish the regularity when $f$ is small and then extend this to the desired result.
\begin{prop} 
\label{prop:caffsilv}
Let $2\leq p<\infty$, $0<s<1$ and $q$ be such that
\[
\left\{\begin{array}{lr}
q>\dfrac{N}{\,p},& \mbox{ if } p\le N,\\
&\\
q\ge 1,& \mbox{ if }p>N,
\end{array}
\right.
\]
We consider $\Theta=\Theta(N,p,q)$ the exponent defined as
$$
\Theta = \min\Big\{1,\frac{p-N/q}{p-1},\frac{sp}{p-1}\Big\}.
$$
For every $0<\varepsilon<\Theta$ there exists $\eta(N,p,q,s,\varepsilon)>0$ such that if $f\in L^q_{\rm loc}(B_4(x_0))$ and
\[
\|f\|_{L^q(B_1(x_0))}\leq \eta,\quad 0\leq A\leq 1,
\] 
then every weak solution $u\in W^{1,p}_{\rm loc}(B_4(x_0))\cap L^{p-1}_{s\,p}(\mathbb{R}^N)$ of the equation
\[
-\Delta_p u+A(-\Delta_p)^s u=f,\qquad \mbox{ in }B_4(x_0),
\]
that satisfy
\begin{equation}
\label{startup}
\|u\|_{L^\infty(B_1(x_0))}\leq 1,\qquad \int_{\mathbb{R}^N\setminus B_1(x_0)}\frac{|u|^{p-1}}{|x|^{N+s\,p}}\, dx\leq 1
\end{equation}
belongs to $C^{\Theta-\varepsilon}(\overline{B_{1/8}(x_0)})$ with the estimate
$$
[u]_{C^{\Theta-\varepsilon}(\overline{B_{1/8}}(x_0))}\leq C(N,s,p,q,\varepsilon).
$$
\end{prop}
\begin{proof} 
Without loss of generality, we may assume that $x_0=0$. We divide the proof in two parts.
\vskip.2cm\noindent
{\bf Part 1: Regularity at the origin}. Here we prove that for every $0<\varepsilon<\Theta$ and every $0<r<1/2$, there exists $\eta$ and a constant $C=C(N,p,q,s,\varepsilon)>0$ such that if $f$ and $u$ are as above, then we have
\[
\sup_{x\in B_r} |u(x)-u(0)|\leq C\,r^{\Theta-\varepsilon}.
\]
Without loss of generality, we assume $u(0)=0$. Fix $0<\varepsilon<\Theta$ and observe that it is sufficient to prove that there exists $\lambda<1/2$ and $\eta>0$ (depending on $N,p,q,s$ and $\varepsilon$) such that if $f$ and $u$ are as above, then
\begin{equation}
\label{eq:keq}
\sup_{B_{\lambda^k}}|u|\leq \lambda^{k\,(\Theta-\varepsilon)},\qquad \int_{\mathbb{R}^N\setminus B_1}\left|\frac{u(\lambda^k\, x)}{\lambda^{k\,(\Theta-\varepsilon)}}\right|^{p-1}\,|x|^{-N-s\,p}\, dx\leq 1,
\end{equation}
for every $k\in\mathbb{N}$. Indeed, assume this is true. Then for every $0<r<1/2$, there exists $k\in \mathbb{N}$ such that $\lambda^{k+1}< r\le \lambda^k$. From the first property in \eqref{eq:keq}, we obtain
\[
\sup_{B_r} |u|\le \sup_{B_{\lambda^k}} |u|\le \lambda^{k\,(\Theta-\varepsilon)}=\frac{1}{\lambda^{\Theta-\varepsilon}}\,\lambda^{(k+1)\,(\Theta-\varepsilon)}\le C\,r^{\Theta-\varepsilon},
\]
as desired.
\par
We prove \eqref{eq:keq} by induction. For $k=0$, \eqref{eq:keq} holds true by the assumptions in \eqref{startup}. 
Suppose \eqref{eq:keq} holds up to $k$, we now show that it also holds for $k+1$, provided that 
\[
\|f\|_{L^q(B_1)}\le \eta,
\]
with $\eta$ small enough, but independent of $k$.
Define
$$
w_k=\frac{u(\lambda^k x)}{\lambda^{k\,(\Theta-\varepsilon)}}.
$$
By the hypotheses
\begin{equation}\label{eq:wkass}
\|w_k\|_{L^\infty(B_1)}\leq 1 \qquad \mbox{ and } \qquad \int_{\mathbb{R}^N\setminus B_1}\frac{|w_k|^{p-1}}{|x|^{N+s\,p}}\,dx\leq 1.
\end{equation}
Moreover
$$
-\Delta_p w_k(x)+A\,\lambda^{kp(1-s)}(-\Delta_p)^s w_k (x) = \lambda^{k\,[p\,-(\Theta-\varepsilon)(p-1)]}\,f(\lambda^k\, x)=:f_k(x), 
$$
so that  
$$\
\|f_k\|_{L^{q}(B_1)}=\lambda^{k(p\,-(\Theta-\varepsilon)(p-1))}\lambda^{-\frac{N}{q}\,k}\,\left(\int_{B_{\lambda^k}}|f|^{q}\,dx\right)^{\frac{1}{q}}\leq \|f\|_{L^{q}(B_1)}\le \eta.
$$
Here we used the hypotheses on $f$ and the definition of $\Theta$, and again the fact that $\lambda<1/2$. 
By Theorem \ref{prop:death!}, we may take $h_k$ to be the weak solution of 
$$
\left\{\begin{array}{rcll}
-\Delta_p h+A\,\lambda^{kp(1-s)}(-\Delta_p)^s h &=& 0,& \mbox{ in } B_1,\\
h&=&w_k,& \mbox{ in } \mathbb{R}^N\setminus B_1.
\end{array}
\right.
$$
By Proposition \ref{prop:Linfty}, we have 
\[
\|w_k-h_k\|_{L^\infty(B_{3/4})}<C\eta^\frac{1}{p-1},\quad C=C(N,p,q).\normalcolor
\]
Then, we have the following estimate
\begin{equation}
\label{estimata}
\begin{split}
|w_k(x)|&\leq |w_k(x)-h_k(x)|+|h_k(x)-h_k(0)|+|h_k(0)-w_k(0)|\\
&\leq  2C\eta^\frac{1}{p-1}+[h_k]_{C^{\Theta-\varepsilon/2}(B_{1/2})}\,|x|^{\Theta-\frac{\varepsilon}{2}},\qquad\qquad \mbox{ for }x\in B_{1/2},\\
\end{split}
\end{equation}
We also used that $h_k$ is $C^{\Theta-\varepsilon/2}$ in $(\overline{B_{1/2}})$ thanks to Theorem \ref{teo:1}, that implies\footnote{Note that Theorem \ref{teo:1} gives an estimate in $B_\frac14$, but by covering $B_\frac12$ with balls of radius $1/4$ this yields an estimate in $B_\frac12$.}
\[
[h_k]_{C^{{\Theta-\varepsilon/2}}(B_{1/2})}\le C\, \left(\|h_k\|_{L^\infty(B_{1})}+\Tail_{p-1,sp,sp}(h_k,0,1)\right)\leq C_1,\quad C_1=C_1(N,p,q,s,\varepsilon).
\]
Here we have observed that the quantities in the right-hand side are uniformly bounded, independently of $k$. Indeed, by the triangle inequality, Proposition \ref{prop:Linfty} and \eqref{eq:wkass} we have
\[
\|h_k\|_{L^\infty(B_{1})}\le  \|h_k-w_k\|_{L^\infty(B_{1})}+\|w_k\|_{L^\infty(B_{1})}\le C\eta^\frac{1}{p-1}+1.
\]
\par
Let
$$
w_{k+1}(x)=\frac{u(\lambda^{k+1}\, x)}{\lambda^{(k+1)\,(\Theta-\varepsilon)}}=\frac{w_k(\lambda\, x)}{\lambda^{\Theta-\varepsilon}}.
$$
By choosing $\eta$ so that $2C\eta^\frac{1}{p-1}<\lambda^\Theta$ and $\lambda$ small enough, we can transfer estimate \eqref{estimata} to $w_{k+1}$. Indeed, we have 
\[
\begin{split}
|w_{k+1}(x)|\leq  2C\eta^\frac{1}{p-1}\,\lambda^{\varepsilon/2-\Theta}+C_1\,\lambda^{\varepsilon/2}|x|^{\Theta-\varepsilon/2}\leq (1+C_1\,|x|^{\Theta-\varepsilon/2})\,\lambda^{\varepsilon/2},\qquad  x\in B_\frac{1}{2\lambda}.
\end{split}
\]
The previous estimate implies in particular that $\|w_{k+1}\|_{L^\infty(B_1)}\leq 1$ for $\lambda$ satisfying
\begin{equation}
\label{uno}
\lambda<\min\left\{\frac{1}{2},(1+C_1)^{-\frac{2}{\varepsilon}}\right\}.
\end{equation}
This information, rescaled back to $u$, is exactly the first part of \eqref{eq:keq} for $k+1$. As for the second part of \eqref{eq:keq}, the upper bound for $|w_{k+1}|$ and the fact that $\Theta<\frac{sp}{p-1}$ imply
\begin{equation}
\label{nonlocal1}
\begin{split}
\int_{B_{\frac{1}{2\lambda}}\setminus B_{1}} \frac{|w_{k+1}|^{p-1}}{|x|^{N+s\,p\,}}\,dx&\leq \lambda^{\varepsilon\, (p-1)/2}\int_{B_{\frac{1}{2\lambda}}\setminus B_{1}} \frac{(1+C_1\,|x|^{\Theta-\varepsilon/2})^{p-1}}{|x|^{N+s\,p}}\,dx\\
&\leq (1+C_1)^{p-1}\,\lambda^{\varepsilon\, (p-1)/2}\int_{B_{\frac{1}{2\lambda}}\setminus B_{1}} \frac{1}{|x|^{N+sp+(\varepsilon/2-\Theta)\,(p-1)}}\,dx\\
&\leq \frac{C_2}{s\,p-(\Theta-\varepsilon/2)\,(p-1)}\,\lambda^{\varepsilon\,(p-1)/2}.
\end{split}
\end{equation}
By a change of variables and using that $|w_k|\leq 1$ in $B_1$, we also see that
\begin{equation}
\label{nonlocal2}
\int_{B_{\frac{1}{\lambda}}\setminus B_{\frac{1}{2\lambda}}} \frac{|w_{k+1}|^{p-1}}{|x|^{N+s\,p\,}}\,dx=\lambda^{(\varepsilon-\Theta)\,(p-1)+s\,p}\,\int_{B_1\setminus B_\frac12} \frac{|w_k(x)|^{p-1}}{|x|^{N+s\,p}}\,dx\leq {C_3\,\lambda^{\varepsilon\,(p-1)/2}}.
\end{equation}
In addition, by the integral bound on $w_k$ in \eqref{eq:wkass}
\begin{equation}
\label{nonlocal3}
\int_{\mathbb{R}^N\setminus B_{\frac{1}{\lambda}}} \frac{|w_{k+1}(x)|^{p-1}}{|x|^{N+s\,p}}\,dx= \lambda^{(\varepsilon-\Theta)\,(p-1)+s\,p}\,\int_{\mathbb{R}^N\setminus B_1}\frac{|w_k(x)|^{p-1}}{|x|^{N+s\,p}}\,dx\leq \lambda^{\varepsilon\, (p-1)/2}.
\end{equation}
In both estimates, we have also used that $\lambda<1/2$ and the fact that
\begin{equation}
\label{eq:spdecay}
(\varepsilon-\Theta)\,(p-1)+s\,p\ge \varepsilon\,\frac{p-1}{2}.
\end{equation}
We observe that the constants $C_2$ and $C_3$ depend on $N,p,q,s$ and $\varepsilon$ only.
From \eqref{nonlocal1}, \eqref{nonlocal2} and \eqref{nonlocal3},
we get that the second part of \eqref{eq:keq} holds, provided that
$$
\left(\frac{C_2}{\varepsilon\,(p-1)}+C_3+1\right)\,\lambda^{\varepsilon\,(p-1)/2}\leq 1.
$$
By taking \eqref{uno} into account, we finally obtain that \eqref{eq:keq} holds true at step $k+1$ as well, provided that $\lambda$ and $\eta$ (depending on $N,p,q,s$ and $\varepsilon$) are chosen so that
\[
\lambda<\min\left\{\frac{1}{2},(1+C_1)^{-\frac{2}{\varepsilon}}, \left(\frac{C_2}{\varepsilon\,(p-1)}+C_3+1\right)^\frac{2}{\varepsilon\,(p-1)}\right\}\qquad \mbox{ and }\qquad 2C\eta^\frac{1}{p-1}<\frac{\lambda^\Theta}{2}.
\]
The induction is complete.
\vskip.2cm\noindent
\textbf{Part 2:} We now show the desired regularity in the whole ball $B_{1/8}$. We choose $0<\varepsilon<\Theta$ and take the corresponding $\eta$, obtained in {\bf Part 1}. Take $z_0\in B_{1/2}$, let $L=2^{N+1}\,(1+|B_1|)$ and define
$$
v(x):=L^{-\frac{1}{p-1}}\,u\left(\frac{x}{2}+z_0\right),\qquad x\in \mathbb{R}^N.
$$
We observe that $v\in W^{1,p}_{\rm loc}(B_4)\cap L^{p-1}_{s\,p}(\mathbb{R}^N)$ and that $v$ is a weak solution in $B_4$ of 
\[
-\Delta_p v(x) +A\,2^{sp-p}(-\Delta_p)^s v(x)=\frac{2^{-p}}{L}\,f\left(\frac{x}{2}+z_0\right)=:\widetilde f(x),
\]
with
\[
\left\|\widetilde f\right\|_{L^{q}(B_1)}=\frac{2^{N/q-p}}{L}\,\|f\|_{L^{q}(B_{\frac12}(z_0))}\le \frac{2^{N/q-p}}{L}\,\eta<\eta.
\]
By construction, we also have
\[
\|v\|_{L^\infty(B_1)}\leq 1,
\] 
and since $B_{1/2}(z_0)\subset B_1$, it follows that
\[
\begin{split}
\int_{\mathbb{R}^N\setminus B_1}\frac{|v(x)|^{p-1}}{|x|^{N+s\,p}}\, dx&=\frac{2^{-s\,p}}{L}\,\int_{\mathbb{R}^N\setminus B_{1/2}(z_0)}\frac{|u(y)|^{p-1}}{|y-z_0|^{N+s\,p}}\,dy\\
&\le \frac{1}{L}\,\left(\frac{1}{2}\right)^{s\,p}\,\left(\frac{1}{1-|z_0|}\right)^{N+s\,p}\,\int_{\mathbb{R}^N\setminus B_1}\frac{|u(y)|^{p-1}}{|y|^{N+s\,p}}\,dy+\frac{2^{N}}{L}\,\|u\|^{p-1}_{L^{p-1}(B_1)}\\
&\leq \frac{2^{N}}{L}\,\int_{\mathbb{R}^N\setminus B_1}\frac{|u(y)|^{p-1}}{|y|^{N+s\,p}}dy+\frac{2^N\,|B_1|}{L}\,\|u\|^{p-1}_{L^\infty(B_1)}\leq 1,
\end{split}
\]
by the definition of $L$ and the hypotheses in \eqref{startup}.
Here we have used Lemma 2.3 in \cite{BLS} with the balls $B_{1/2}(z_0)\subset B_1$. We may therefore apply {\bf Part 1} to $v$ and obtain
$$
\sup_{x\in B_r}|v(x)-v(0)|\leq C\,r^{\Theta-\varepsilon},\quad 0<r<\frac{1}{2}.
$$
In terms of $u$ this is the same as
\begin{equation}
\label{supest}
\sup_{x\in B_r(z_0)}|u(x)-u(z_0)|\leq C\,L^\frac{1}{p-1}\,r^{\Theta-\varepsilon},\qquad 0<r<\frac{1}{4}.
\end{equation}
We note that this holds for any $z_0\in B_{1/2}$. Now take any pair $x,y\in B_{1/8}$ and set $|x-y|= r$. We observe that $r<1/4$ and we set $z=(x+y)/2$. Then we apply \eqref{supest} with $z_0=z$ and obtain
\[
\begin{split}
|u(x)-u(y)|\leq |u(x)-u(z)|+|u(y)-u(z)|&\leq 2\sup_{w\in B_r(z)}|u(w)-u(z)|\\
&\leq 2\,C\,L^\frac{1}{p-1}\,r^{\Theta-\varepsilon}=2\,C\,L^\frac{1}{p-1}\,|x-y|^{\Theta-\varepsilon},
\end{split}
\]
which is the desired result.
\end{proof}

We are now in the position to prove Theorem \ref{nonthm}.
\begin{proof}[~Proof of Theorem \ref{nonthm}]
We may assume $x_0=0$ without loss of generality. We modify $u$ so that it fits into the setting of Proposition \ref{prop:caffsilv}. We choose $0<\delta<\Theta$, take $\eta$ as in Proposition \ref{prop:caffsilv}  with the choice $\varepsilon=\Theta-\delta$ and set
\[
\mathcal{A}_R=\|u\|_{L^\infty(B_{R})}+\left(R^{s\,p}\,\int_{\mathbb{R}^N\setminus B_{R}}\frac{|u(y)|^{p-1}}{|y|^{N+s\,p}}\,  dy\right)^\frac{1}{p-1}+\left(\frac{R^{p-N/q}\|f\|_{L^{q}(B_{R})}}{\eta}\right)^\frac{1}{p-1}.
\] 
By scaling arguments, it is sufficient to prove that the rescaled function
\[
u_R(x):=\frac{1}{\mathcal{A}_R}\,u(R\,x),\qquad \mbox{ for }x\in B_4,
\]
satisfies the estimate
\[
[u_R]_{C^{\delta}(B_{1/8})}\leq C.
\]
It is easily seen that the choice of $\mathcal{A}_R$ implies
$$
\|u_R\|_{L^\infty(B_1)}\leq 1,\qquad \int_{\mathbb{R}^N\setminus B_1}\frac{|u_R|^{p-1}}{|x|^{N+s\,p}}\, dx\leq 1. 
$$
In addition, $u_R$ is a weak solution of
$$
-\Delta_p u_R\,(x)+A\,R^{p-sp} u_R(-\Delta_p)^s u_R\, (x) = \frac{R^{p}}{\mathcal{A}_R^{p-1}}\,f(R\,x):= f_R(x),\qquad x\in B_4,
$$
with $\|f_R\|_{L^{q}(B_{1})}\le \eta$ and $R^{p-sp}<1$. We may therefore apply Proposition \ref{prop:caffsilv} with $\varepsilon=\Theta-\delta$ to $u_R$ and obtain
\[
[u_R]_{C^{\delta}(B_{1/8})}\leq C.
\]
This concludes the proof.
\end{proof}

\appendix
\section{Pointwise inequalities}\label{sec:app}

In this section, we list the pointwise inequalities used throughout the whole paper.

The following result can be found in \cite[page 97, Inequality (I)]{PLin}.
\begin{Lemma}\label{lem:pineq}
For $a,b\in \R^N$ and $p\ge 2$, we have
$$
\langle |a|^{p-2}a-|b|^{p-2},a-b\rangle \geq 2^{2-p}|a-b|^p.
$$
\end{Lemma}
For the following result, see \cite[page 99, Inequality (V)]{PLin}.
\begin{Lemma}\label{in1}
Let $a,b\in\mathbb{R}^N$. Then for any $p\ge 2$, we have
\begin{equation}\label{V}
\langle|b|^{p-2}b-|a|^{p-2}a,b-a\rangle\geq \frac{4}{p^2}\Big||b|^\frac{p-2}{2}b-|a|^\frac{p-2}{2}a\Big|^2.
\end{equation}
\end{Lemma}
For the following inequality, see \cite[page 100, Inequality (VI)]{PLin}.
\begin{Lemma}\label{in2}
Let $a,b\in\mathbb{R}^N$. Then, for any $p\ge 2$, we have
\begin{equation}\label{VI}
\Big||b|^{p-2}b-|a|^{p-2}a\Big|\leq(p-1)\Big(|b|^\frac{p-2}{2}+|a|^\frac{p-2}{2}\Big)\Big||b|^\frac{p-2}{2}b-|a|^\frac{p-2}{2}a\Big|.
\end{equation}
\end{Lemma}

The following is Lemma A.5 in \cite{BLS}.
\begin{Lemma}\label{lem:powertest}
Let $p\ge 2$, $\gamma\ge 1$ and $a,b,c,d\in\mathbb{R}$. Then we have
\begin{equation}
\label{erik}
\begin{split}
\Big(J_p(a-c)-J_p(b-d)\Big)&\Big(J_{\gamma+1}(a-b)-J_{\gamma+1}(c-d)\Big)\\
&\ge \frac{1}{C}\, \Big||a-b|^\frac{\gamma-1}{p}\,(a-b)-|c-d|^\frac{\gamma-1}{p}\,(c-d)\Big|^p,
\end{split}
\end{equation}
for some $C=C(p,\gamma)>0$.
\end{Lemma}



\paragraph{Data availability} 
Data sharing not applicable to this article as no datasets were generated or analysed during the current study.

\noindent {\textsf{Prashanta Garain\\  Department of Mathematics\\ Uppsala University\\
Box 480, 751 06 Uppsala, Sweden}  \\
Current Affiliation: Department of Mathematics\\
Indian Institute of Technology Indore,\\
Khandwa Road, Simrol,\\
Indore 453552, India\\ 
\textsf{e-mail}: pgarain92@gmail.com\\

\noindent {\textsf{Erik Lindgren\\  Department of Mathematics, KTH -- Royal Institute of Technology\\ 100 44, Stockholm, Sweden}  \\
\textsf{e-mail}: eriklin@kth.se\\

\end{document}